\documentclass[11pt]{amsart}

\usepackage{amsmath}
\usepackage{amssymb}
\usepackage{bbm}
\usepackage{pdfsync}
\usepackage{esint} 

\newcommand{\C}{{\mathbb C}}       
\newcommand{\R}{{\mathbb R}}       
\newcommand{\N}{{\mathbb N}}

\newcommand{\Z}{{\mathbb Z}}       

\newcommand{\HH}{{\mathcal H}}

\newcommand{\diam}{{\rm diam}}
\newcommand{\dist}{{\rm dist}}

\newcommand{\rf}[1]{{(\ref{#1})}}

\newcommand{\supp}{\operatorname{supp}}

\newcommand{\vphi}{{\varphi}}
\newcommand{\ve}{{\varepsilon}}
\newcommand{\vv}{{\vspace{2mm}}}
\newcommand{\vvv}{{\vspace{3mm}}}
\newcommand{\wt}[1]{{\widetilde{#1}}}

\newcommand{\bx}{\bar x}
\newcommand{\by}{\bar y}

\newcommand{\noi}{\noindent}

\def\Xint#1{\mathchoice
{\XXint\displaystyle\textstyle{#1}}%
{\XXint\textstyle\scriptstyle{#1}}%
{\XXint\scriptstyle\scriptscriptstyle{#1}}%
{\XXint\scriptscriptstyle\scriptscriptstyle{#1}}%
\!\int}
\def\XXint#1#2#3{{\setbox0=\hbox{$#1{#2#3}{\int}$ }
\vcenter{\hbox{$#2#3$ }}\kern-.58\wd0}}

\def\avint{\;\Xint-}

\usepackage{pgf,tikz} 

\definecolor{ffffff}{rgb}{1.0,1.0,1.0}
\definecolor{qqqqff}{rgb}{0.0,0.0,1.0}
\definecolor{ffqqqq}{rgb}{1.0,0.0,0.0}
\definecolor{zzzzqq}{rgb}{0.6,0.6,0.0}
\definecolor{marronet}{rgb}{0.6,0.2,0}
\definecolor{Negre}{rgb}{0,0,0}
\definecolor{vermell}{rgb}{0.8,0.05,0.05}
\definecolor{blau}{rgb}{0.3,0.2,1.}
\definecolor{blauclar}{rgb}{0.,0.,1.}
\definecolor{grisfosc}{rgb}{0.25098039215686274,0.25098039215686274,0.25098039215686274}
\definecolor{verd}{rgb}{0.1,0.6,0.1}
\definecolor{taronja}{rgb}{0.9,0.6,0.05}
\definecolor{vermellclar}{rgb}{1.,0.,0.}
\definecolor{verdet}{rgb}{0,0.8,0.1}
\definecolor{blauverd}{rgb}{0,0.4,0.2}
\definecolor{grisclar}{rgb}{0.6274509803921569,0.6274509803921569,0.6274509803921569}

\newtheorem{theorem}{Theorem}[section]
\newtheorem{lemma}[theorem]{Lemma}

\newtheorem{coro}[theorem]{Corollary}
\newtheorem{propo}[theorem]{Proposition}

\newtheorem*{theorem*}{Theorem}

\theoremstyle{definition}

\theoremstyle{remark}
\newtheorem{rem}[theorem]{Remark}

\numberwithin{equation}{section}

\newcommand{\brem}{\begin{rem}}
\newcommand{\erem}{\end{rem}}

\begin{document}

\title[Removable singularities for solutions of the fractional Heat equation]{Removable singularities for solutions of the fractional Heat equation in time varying domains}

\author{Joan Mateu}
\address{Joan Mateu\\
Departament de Matem\`atiques 
\\
 Universitat Aut\`onoma de Barcelona
 \\
08193 Bellaterra (Barcelona), Catalonia.
}
\email{mateu@mat.uab.cat}
\author{Laura Prat}
\address{Laura Prat\\
Departament de Matem\`atiques 
\\
Universitat Aut\`onoma de Barcelona
\\
08193 Bellaterra (Barcelona), Catalonia.
\\
}
\email{laurapb@mat.uab.cat
}

\thanks{ L.P. was supported by PID2020-114167GB-I00 (MINECO, Spain) and J.M. was supported by PID2020-112881GB-I00 (MINECO, Spain).}

\begin{abstract}
 In this paper we study  removable singularities for solutions of the fractional heat equation in time varying domains. We introduce associated capacities and we study some of its metric and geometric properties.
\end{abstract}

\maketitle

\section{Introduction}

In the present paper we study  removable singularities for solutions of the fractional heat equation in time varying domains.  Our main motivation comes from the paper \cite{mateuprattolsa} where we studied removable singularities for regular $(1,1/2)-$Lipschitz solutions of the heat equation in time varying domains. Our ambient space is $\R^{N+1}$ with a generic point denoted as $\bar x=(x,t)\in\R^{N+1}$, where $x\in\R^N$ and $t\in\R$.  For a smooth function $f$ depending on $(x,t)\in\R^{N+1}$, the heat equation is just \begin{equation}\label{gaussian}\Theta(f)=\Delta f+\partial_tf=0,\end{equation} with fundamental solution $\displaystyle W(x,t)=(4\pi t)^{-N/2}\exp(-|x|^2/4t)$ for $t>0$ and $0$ if $t\leq 0$.

 We set $0<s<1$ and let $\Theta^s$ denote the fractional heat operator, 
$\Theta^s=(-\Delta)^s+\partial_t.$
Then, for a smooth function $f$ depending on $(x,t)\in\R^{N+1}$,
$$\Theta^s(f) = (-\Delta)^s f +\partial_t f=0$$
is just the fractional heat equation.  The pseudodifferential operator $(-\Delta)^s=(-\Delta_x)^s$, $0<s<1$, is the fractional Laplacian with respect to the $x$ variable. It may be defined through its Fourier transform,  $\displaystyle \widehat{(-\Delta)^s }f(\xi)=|\xi|^{2s}\hat f(\xi)$, or by its representation $$(-\Delta)^sf(x)=c(n,s)\int_{\R^N}\frac{f(x)-f(y)}{|x-y|^{N+2s}}dy,$$ where $c(n,s)$ is a normalization constant (see \cite{landkof} and \cite{stein} for its basic properties). The standard Laplace operator, $-\Delta$, is recovered when taking the limit $s\to 1$ (see \cite[Section 4]{npv}), but there is a big difference between the local operator $-\Delta$, that appears in the classical heat equation and represents Brownian motion,  and the non-local family $(-\Delta)^s$, $0<s<1$, which are generators of L\'evy processes in Stochastic PDEs (see \cite{4} i \cite{6}). 

Given $\bar x= (x,t)$ and $\bar y = (y,u)$, with $x,y\in\R^N$, $t,u\in\R$,
we consider the 
$s$-parabolic distance in $\R^{N+1}$ defined by
$$\dist_p(\bar x,\bar y) = \max\big(|x-y|, \,|t-u|^{1/{2s}}\big).$$
Sometimes we also write
$|\bar x - \bar y|_p$ instead of $\dist_p(\bar x,\bar y)$. Notice that for $s=1/2$,  dist$_p$ is comparable to the usual distance.
We denote by $B_p(\bar x,r)$ an $s$-parabolic ball (i.e., in the distance $\dist_p$) centered at $\bar x$ with radius $r$. By an $s$-parabolic cube $Q$ of side length $\ell$, we mean a set of the form
$$I_1 \times\ldots\times I_N\times I_{N+1},$$
where $I_1,\ldots,I_N$ are intervals in $\R$ with length $\ell$, and $I_{N+1}$ is another interval with
length $\ell^{2s}$. We write $\ell(Q)=\ell$. 

We say that a Borel measure $\mu$ in $\R^{N+1}$ has upper parabolic growth of degree $N+2s-1$, $0<s<1$, if there
exists some constant $C$ such that
\begin{equation}\label{C=1}
\mu(B_p(\bar x,r))\leq C r^{N+2s-1}\quad\mbox{ for all $\bar x\in\R^{N+1},\, r>0$.}
\end{equation}
Clearly, this is equivalent to saying that any $s$-parabolic cube $Q\subset\R^{N+1}$ satisfies $\mu(Q)\leq C'
\ell(Q)^{N+2s-1}$. 
Now, let $s=1/2$.
We say that a compact set $E\subset \R^{N+1}$ is removable for the $1/2-$fractional heat equation (or removable for bounded $1/2-$caloric functions)  if any bounded function $f:\R^{N+1}\to\R$, i.e. $\displaystyle\|f\|_{L^\infty(\R^{N+1})}<\infty$, 
satisfying the $\frac12-$fractional heat equation in $\mathbb R^{N+1}\setminus E$, also satisfies the $1/2-$fractional heat equation in the whole space $\R^{N+1}$. 

Given a compact set $E\subset \R^{N+1}$, we define its $1/2-$fractional caloric capacity by
\begin{equation}\label{def1}
\gamma^{1/2}_\Theta(E) =\sup\{|\langle\nu,1\rangle|:  \nu\in\mathcal D', \,\supp\nu\subset E \mbox{ and } \|P*\nu\|_{L^\infty(\R^{N+1})} \leq 1\},
\end{equation}
where  $\mathcal D'$ is the space of distributions in $\R^{N+1}$ and
$P(x,t)$ denotes the fundamental solution of the $1/2-$fractional heat equation in  $\R^{N+1}$, that  is 
\begin{equation}\label{kernel}
P(x,t) =\left\{\begin{array}{l} \cfrac{t}{(t^2+|x|^2)^{\frac{N+1}2}}\quad \mbox{if }\; t>0\\\\\quad0\hspace{2.65cm}\mbox{if }\;t\leq 0\end{array}\right..
\end{equation}
The fractional heat equation can be solved (by applying Fourier transform) for all $0<s\leq 1$ by means of the fundamental solution $P_s(x,t)$, which is the inverse transform of the function $e^{-|\xi|^{2s}t}$. It is worth mentioning that only in the particular cases $s=1$ and $s=1/2$ the kernel is known to be explicit. In the first case we get the Gaussian kernel for the standard heat equation $\displaystyle W(x,t)=(4\pi t)^{-N/2}e^{-|x|^2/4t}$ and in the fractional case $s=1/2$ we have \eqref{kernel}. For $0<s<1$, there is no explicit fundamental solution for the $s-$fractional heat equation. In the 1960s, the probabilists Blumenthal and Getoor \cite[Theorem 2.1]{blumenthalgetoor} generalized in a precise way the power-like tail behaviour observed in the case $s=1/2$ to the other values of $0<s<1$, that is, they established that the fundamental solution, $P_s(x,t)$, of the fractional heat equation, for $0<s<1$, satisfies $$ P_s(x,t)\approx\frac{t}{(t^{1/s}+|x|^2)^{\frac{N+2s}2}}$$ when $ t>0$ (and $0$ for $t\leq 0$). Here $a\approx b$ means that $a/b$ is uniformly bounded above and below by a constant. Notice the marked difference with the Gaussian kernel $W(x,t)$ of the heat equation \eqref{gaussian} (case $s=1$). The behaviour as $x$ tends to infinity of $P_s$ is power-like while $W$ has exponential spatial decay. 

In the paper, we will also introduce and study some $s$-fractional caloric capacities $\gamma_\Theta^s$, $0<s<1$, associated with the kernels $P_s$ of the $s$-fractional heat equation.

We shall now give a brief description of the main results in the paper. 
Section \ref{one} includes some estimates of the kernels $P_s(x,t)$ ($0<s<1$) and its derivatives, that will be needed in the rest of the paper. The four next sections deal with the case $s=1/2$. More concretely, in Section \ref{section-localization} we prove a localization result for $s=1/2$, that is, for a distribution $\nu$, we localize the potential  $P*\nu$ in the $L^\infty-$norm. The localization method for the Cauchy potential $\nu*1/z$ in the plane is a basic tool 
 developed by A.G. Vitushkin in the theory of rational approximation in the plane. This was later adapted in \cite{paramonov} for Riesz potential $\nu*x/|x|^{N}$ in $\R^N$ and used in problems of $\mathcal{C}^1-$harmonic approximation. These localization results have also been essential to prove the semiadditivity of analytic capacity and of Lipschitz harmonic capacity, see \cite{tolsasemiad} and \cite{volbergsemiad} respectively (see  also \cite{tams} for other related capacities).
 
 In Section \ref{section-capacities} we study the connection between $\frac12-$fractional caloric removability and the capacity $\gamma_\Theta^{1/2}$.  In particular, we show that
 a compact set $E\subset\R^{N+1}$ is $\frac12-$fractional caloric removable if and only if $\gamma_\Theta^{1/2}(E)=0$. 
  We also compare the capacity $\gamma_{\Theta}^{1/2}$ to the Hausdorff content $\mathcal{H}_{\infty}^{N}$ and we prove that if $E$ has zero $N-$dimensional Hausdorff measure, i.e., $\HH^{N}(E)=0$, then  $\gamma_{\Theta}^{1/2}(E)=0$ too. In the converse direction, we show that if $E$ has Hausdorff dimension larger than $N$, then $\gamma_{\Theta}^{1/2}(E)$ is positive. Hence, the critical dimension for $\frac12-$fractional caloric capacity 
 (and thus for $\frac12-$fractional caloric removability)
 occurs in dimension $N$, in accordance with the classical case.  Also by analogy with the classical case, one should expect that $\gamma_\Theta^{1/2}(E)>0$ if and only if $\gamma_{\Theta,+}^{1/2}(E)>0$ or even that both capacities are comparable (the definition of $\gamma_{\Theta,+}^{1/2}$ is the same as \rf{def1} but instead of distributions one considers positive measures). However, there is a big obstacle when trying to follow this approach. Namely, the kernel
$P$ is not antisymmetric and thus, if $\nu$ is such that $T\nu=P*\nu$ is in $L^\infty(\R^{N+1})$, apparently one cannot get any useful information regarding $T^*\nu$ ($T^*$ being the dual operator to $T$). This prevents any direct application of the usual $T1$ or $Tb$ theorems from Calder\'on-Zygmund theory, which are  essential tools in the case of analytic capacity and Lipschitz harmonic capacity. Hence, In Section \ref{section-capacities}, due to the lack of antisymmetry of our kernel $P$, we also introduce a new capacity $\wt\gamma_{\Theta,+}^{1/2}$.  We set
$\wt\gamma_{\Theta,+}(E)=\sup\mu(E),$
the supremum taken over all positive measures supported on $E$ with
$\displaystyle\|T\mu\|_{L^\infty(\R^{N+1})} \leq 1$ and $\|T^*\mu\|_{L^\infty(\R^{N+1})} \leq 1.$ Clearly $\gamma_\Theta^{1/2}\geq\gamma_{\Theta,+}^{1/2}\geq \wt\gamma_{\Theta,+}^{1/2}$. We show that the capacity $\wt\gamma_{\Theta,+}^{1/2}$ can be characterized in terms of the $L^2$-norm of $T$. 

In Sections \ref{section-removablesets} and \ref{section-comparability}, we give some concrete examples of $\frac12-$caloric removable and non-removable sets with Hausdorff dimension $N$. In particular, in the first section we construct a self similar Cantor set in $\R^{N+1}$ with positive and finite Hausdorff $N$-dimensional measure which is $\frac12-$caloric removable and in the second one we show that on the plane, the capacity $\gamma_\Theta^{1/2}$ vanishes on vertical segments and is positive on horitzontal ones. These examples allow us to deduce that, in dimension two, neither analytic capacity nor newtonian capacity are comparable to the $1/2-$fractional caloric capacity $\gamma_\Theta^{1/2}$ (although these three capacities have the same critical dimension). 

Sections \ref{section-s1} and \ref{section-s2} are devoted to the study of the $s-$fractional capacities $\gamma_\Theta^s$, when $1/2<s<1$ and $0<s<1/2$ respectively. In the first case, namely $1/2<s<1,$ we are able to show that the critical dimension for the $\gamma_\Theta^s$-capacity is $N+2s-1$. We have had several technical problems when trying to prove this statement for the case $0<s<1/2$ but we can show that sets $E$ with zero $(N+2s-1)-$dimensional Hausdorff measure, i.e., $\HH^{N+2s-1}(E)=0$, have $\gamma_{\Theta}^{s}(E)=0$.

\vv
Some comments about the notation used in the paper: as usual, the letter $C$ stands for an absolute constant which may change its value at different occurrences.
The notation $A\lesssim B$ means that there is a positive absolute constant $C$ such that $A\leq CB$. Also, $A\approx B$ is equivalent to $A\lesssim B\lesssim A$.
The gradient symbol $\nabla$ refers to $(\nabla_x,\partial_t)$, with $x\in\R^N$ and $t\in\R$.

\section{Some  preliminary estimates on the kernels $P$ and $P_s$}\label{one}

In the next two lemmas we will obtain upper bounds for the kernels $P(x,t)$, $P_s(x,t)$ and its derivatives. We will need them later.

\begin{lemma}\label{bound}
For any $\bar x=(x,t)\in\R^{N+1}$, $x\in\R^N$ and $t\in\R$, the following holds:
\begin{itemize}
\item[(a)] $P(\bar x)\lesssim \dfrac1{|\bar x|^{N}}$\; for all $\bar x\neq0$.

\item[(b)] $|\nabla_x P(\bar x)|\lesssim \dfrac1{|\bar x|^{N+1}}$ for all $\bar x\neq0$\;and \;
$|\partial_t P(\bar x)|\lesssim \dfrac1{|\bar x|^{N+1}}$ for all $\bar x\neq(x,0)$.

\item[(c)] For all $\bar x,\bar x'\in \R^{N+1}$ such that $|\bar x- \bar x'|\le |\bar x|/2$, $\bar x\neq0$,
$$|P(\bar x) - P(\bar x')|\lesssim \frac{|\bar x - \bar x'|}{|\bar x|^{N+1}}.$$
\end{itemize}
\end{lemma}

Notice that the kernel P is not differentiable with respect to $t$ at the points $(x,0),\;x\in\R^N$.

\begin{proof}
The estimate in (a) comes from the explicit expression of P. For the estimates in (b) we compute
$$
\nabla_xP(x,t)=-\frac{(N+1)xt}{(|x|^2+t^2)^{\frac{N+3}{2}}}, 
$$
and so we get $|\nabla_x P(\bar x)|\lesssim \dfrac1{|\bar x|^{N+1}}.$ On the other hand, for $t\neq 0$
$$
\partial_tP(x,t)=\frac{1}{(|x|^2+t^2)^{\frac{N+1}{2}}}-\frac{(N+1)t^2}{(|x|^2+t^2)^{\frac{N+3}{2}}}. 
$$
Then $|\partial_t P(\bar x)|\lesssim \dfrac1{|\bar x|^{N+1}}$ for all $\bar x\neq(x,0),\, x\in \mathbb R^N $.

 Finally, (c) will be obtained using  (b). Indeed,
given $\bar x,\bar x'\in \R^{N+1}$ such that $|\bar x- \bar x'|\le |\bar x|/2$, write
$$\bar x = (x,t),\quad \bar x' = (x',t'),\quad \hat x = (x',t).$$
Then 
\begin{equation}\label{kernelestimate}
|P(\bar x) - P(\bar x')| \leq |P(\bar x) - P(\hat x)| + |P(\hat x) - P(\bar x')|
\end{equation}

The first term in the above inequality is bounded by 
$$|x-x'|\,\sup_{y\in[x,x']} |\nabla_x P(y,t)|\lesssim  
\frac{| x - x'|}{|\bar x|^{N+1}} \leq \frac{|\bar x -\bar x'|}{|\bar x|^{N+1}}.$$

For the second term in \eqref{kernelestimate} let $t>t'$. If $t'>0$, then
$$
|t-t'|\,\sup_{s\in[t',t]} |\partial_t P(x',s)| 
\lesssim  
  \frac{|t - t'|}{|\bar x|^{N+1}} \leq \frac{|\bar x -\bar x'|}{|\bar x|^{N+1}}.
$$

If $t<0$, $|P(\hat x) - P(\bar x')|=0$ and if $t>0$ and $t'<0$, then 
$$|P(x',t)-P(x',t')|=|P(x',t)-P(x',0)|\lesssim |t|\sup_{s\in[0,t]} |\partial_t P(x',s)| 
\lesssim \frac{|\bar x -\bar x'|}{|\bar x|^{N+1}}.
$$

\end{proof}

\vv

The following lemma shows some growth properties of our kernel $P_s$.
 \begin{lemma}\label{bounds}
 Let $s\in (0,1)$. For any $\bar x=(x,t)\in\R^{N+1}$, $x\in\R^N$ and $t\in\R$, the following holds:
 \begin{enumerate}
 \item $\displaystyle |P_s(\bar x)|\lesssim \frac1{|\bar x|_p^N}$ for all $\bar x\neq 0$. 
 \item For $\alpha\in (0,1)$, $\displaystyle |(-\Delta)^{\alpha}P_s(\bar x)|\lesssim\frac1{|\bar x|_p^{N+2\alpha}}$.
 \item $\displaystyle |\partial_t P_s(\bar x)|\lesssim\frac1{|\bar x|_p^{N+2s}}$ for all $\bar x\neq (x,0)$. 
 \item $\displaystyle |\nabla_x P_s(\bar x)|\lesssim\frac1{|\bar x|_p^{N+1}}$ for all $\bar x\neq 0$.
 \item For $s\in (\frac12,1)$, $\displaystyle |\partial_t^{1-\frac1{2s}} P_s(\bar x)|\lesssim\frac1{|x|^{N-1}|\bar x|_p^{2s}}.$ 
 \end{enumerate}
 \end{lemma}
 \begin{proof}
 The first property follows from the definition of the kernel. For $t\leq 0$, $P_s(x,t)=0$ and for $t>0$,
 
 $$|P_s(\bar x)|=|P_s(x,t)|\approx\frac{t}{(t^{\frac1{s}}+|x|^2)^{\frac{N+2s}2}}\leq\frac{(t^\frac1s+|x|^2)^s}{(t^{\frac1{s}}+|x|^2)^{\frac{N+2s}2}}=\frac1{|\bar x|_p^N}.$$
 
 We prove now the second one. Applying Fourier transform to the fractional heat equation with respect to the space variable $x$, and calling the new variable $\xi$, we get the equation $\displaystyle \partial_t\widehat u=-|\xi|^{2s}\widehat u$, that allows to solve the initial-value
problem in Fourier space by means of the formula $\widehat u(\xi,t)=\widehat u_0(\xi)e^{-|\xi|^{2s}t}$.
Applying the inverse transform, the fractional heat equation can be solved for all $0<s<1$
by means of the fundamental solution, $P_s(x,t)$, which is the inverse transform of the function
$e^{-|\xi|^{2s}t}$. Hence the kernel $P_s$ has Fourier transform $\displaystyle\widehat P_s(\xi,t)=e^{-|\xi|^{2s}t}$. It is well known that for $t>0$ (recall that for $t\leq 0,\; P_s(x,t)=0$) it has the self-similar form \
 
 \begin{equation}\label{selfsimilarphi}
 P_s(x,t)=t^{\frac{-N}{2s}}\phi\left(|x|t^{-\frac1{2s}}\right)
 \end{equation}
 
  for some positive $\mathcal C^{\infty}$ function $\phi$, radially decreasing and satisfying $\phi(u)\approx (1+u^2)^{-\frac{N+2s}{2}}$ (see \cite{blumenthalgetoor} and \cite{vazquez}). 

  Using \eqref{selfsimilarphi} we deduce that 
  
  $$(-\Delta)^{\alpha}P_s(x,t)=t^{-\frac{N}{2s}-\frac{\alpha}{s}}\psi\left(|x|t^{-\frac1{2s}}\right)$$ 
 
  where $\displaystyle \psi(z)=(-\Delta)^{\alpha}\phi(z)$. Since $\widehat\phi(\xi)=e^{-|\xi|^{2s}}$, then $\displaystyle \widehat\psi(\xi)=|\xi|^{2\alpha}e^{-|\xi|^{2s}}$ and using the expression of the inverse Fourier transform of a radial function (see \cite[Section B.5]{grafakos} or \cite[Section IV.I]{steinweiss} for a proof) we obtain 
  
  $$\psi(|z|)=c_{N}|z|^{1-\frac{N}2}\int_0^\infty e^{-r^{2s}}r^{\frac{N}2+2\alpha}J_{\frac{N}2-1}(r|z|)dr,$$
 where $J_k$ is the classical Bessel function of order $k$.
 Thus \cite[Lemma 1]{pruitttaylor} gives us the decay $\displaystyle \left|\psi(|z|)\right|=O(|z|^{-N-2\alpha})$, for $|z|$ large. Since $\psi$ is bounded we have 
 
 \begin{equation}\label{boundpsi}
 \left|\psi(|z|)\right|\lesssim (1+|z|^2)^{\frac{-N-2\alpha}2},
 \end{equation} 
 
 which implies 
 $$\left|(-\Delta)^{\alpha}P_s(x,t)\right|\lesssim\frac1{t^{\frac{N+2\alpha}{2s}}(1+|x|^2t^{-\frac1{s}})^{\frac{N+2\alpha}{2}}}=\frac1{|\bar x|_p^{N+2\alpha}},$$
 the second estimate in the statement of the lemma.
 Observe that for $\alpha=s$ this is
 \begin{equation}\label{bounddeltaphi}
  |(-\Delta)^{s}P_s(\bar x)|\lesssim |\bar x|_p^{-N-2s}. 
 \end{equation}
 Therefore using that $P_s$ is the fundamental solution of the $s$-fractional heat equation
   \begin{equation}
  \label{equationphi}
  \partial_tP_s(x,t)=-(-\Delta)^{s}P_s(x,t)=-t^{-\frac{N}{2s}-1}(-\Delta)^{s}\phi(z),\quad z=|x|t^{-\frac1{2s}},\quad t>0, 
  \end{equation}
 and  the fact that \begin{equation}\label{scas}|(-\Delta)^{s}\phi(|z|)|\lesssim (1+|z|^2)^{\frac{-N-2s}2}\end{equation} (which is estimate \eqref{boundpsi} with $\alpha=s$) we get 
  $$|\partial_tP_s(\bar x)|\lesssim\frac1{t^{1+\frac{N}{2s}}(1+|x|^2t^{-\frac1s})^{\frac{N+2s}2}}=\frac1{|\bar x|_p^{N+2s}},$$
  which is the third estimate in the statement of the lemma.
  
   Next we will estimate the spatial derivative $\nabla_xP_s(x,t)$. Clearly
  $$\nabla_xP_s(x,t)=\nabla_x\left(t^{\frac{-N}{2s}}\phi(|x|t^{-\frac1{2s}})\right)=t^{\frac{-N-1}{2s}}\phi'(|x|t^{-\frac1{2s}}).$$
   If we can show that 
   \begin{equation}\label{fi}\phi'(u)\lesssim (1+u^2)^{-\frac{N+1}2},
   \end{equation} then we are done because 
  $$|\nabla_xP_s(x,t)|\lesssim\frac{t^{\frac{-N-1}{2s}}}{(1+|x|^2t^{-1/s})^{\frac{N+1}2}}=\frac1{|\bar x|_p^{N+1}}.$$
  In order to estimate $\nabla\phi(z)$, we consider the equation for $\phi$ that comes from \eqref{equationphi}, that is
  $$2s(-\Delta)^{s}\phi(z)-N\phi(z)-z\cdot\nabla\phi(z)=0.$$
  Notice that it implies that $$|\nabla\phi(z)|\lesssim\frac{|\phi(z)|+|(-\Delta)^s\phi(z)|}{|z|}.$$
  Since $\phi'(|z|)$ is bounded, from  $\phi(u)\approx (1+u^2)^{-\frac{N+2s}{2}}$ and \eqref{scas} we deduce \eqref{fi}.   
  
 To prove the last estimate in the statement of the lemma, that is
 \begin{equation}\label{partialfractional}
 |\partial_t^{1-\frac1{2s}} P_s(\bar x)|\lesssim\frac1{|x|^{N-1}|\bar x|^{2s}},\quad s\in(1/2,1),
 \end{equation}
   we claim that for $t>0$,
 \begin{equation}\label{expression}
 P_s(x,t)=\frac1{|x|^N}F_s\left(\frac{t}{|x|^{2s}}\right)
 \end{equation} 
 for some function $\displaystyle F_s(u)\approx\frac{u}{(1+u^{1/s})^{\frac{N+2s}2}}$. \eqref{expression} can be proved using \eqref{selfsimilarphi}, that is $P_s(x,t)=t^{\frac{-N}{2s}}\phi\left(|x|/t^{\frac1{2s}}\right)$ for some $\phi$ with $\phi(u)\approx (1+u^2)^{-\frac{N+2s}{2}}$. In fact,
 $$P_s(x,t)=\frac1{|x|^N}\left(\frac{|x|}{t^{\frac1{2s}}}\right)^N\phi(\left(\frac{|x|}{t^{\frac1{2s}}}\right)=\frac1{|x|^N}\left(\frac{t}{|x|^{2s}}\right)^{-\frac{N}{2s}}\phi\left(\Big(\frac{t}{|x|^{2s}}\Big)^{-\frac1{2s}}\right)=\frac1{|x|^N}F_s\left(\frac{t}{|x|^{2s}}\right)$$
 the last equality being a definition for $F_s$. Hence
 $$F_s(u)=u^{\frac{-N}{2s}}\phi(u^{\frac{-1}{2s}})\approx\frac{ u^{\frac{-N}{2s}}}{(1+u^{-1/s})^{\frac{N+2s}{2}}}=\frac{u}{u^{\frac{N+2s}{2s}}(1+u^{-1/s})^{\frac{N+2s}{2}}}=\frac{u}{(1+u^{1/s})^{\frac{N+2s}2}}$$ and claim \eqref{expression} is proved.
 Notice that 
 we have
 $$\partial_t^{1-\frac1{2s}}P_s(x,t)=\frac1{|x|^N}\left[\partial_t^{1-\frac1{2s}}F_s\left(\frac{\cdot}{|x|^{2s}}\right)\right](t)=\frac1{|x|^{N+2s-1}}\partial_t^{1-\frac1{2s}}F_s\left(\frac{t}{|x|^{2s}}\right).$$
 Hence, to show \eqref{partialfractional} it is enough to see that for all $t\in\R$, \begin{equation}\label{bound1}|\partial_t^{1-\frac1{2s}}F_s(t)|\lesssim \min (1,|t|^{-1}).\end{equation}
 Once \eqref{bound1} is available we obtain \eqref{partialfractional} easily: $$|\partial_t^{1-\frac1{2s}} P_s(x,t)|\lesssim\frac1{|x|^{N+2s-1}}\min\left(1,\frac{|x|^{2s}}{t}\right)=\frac1{\max(|x|^{2s},t)|x|^{N-1}}=\frac1{|x|^{N-1}|\bar x|_p^{2s}}.$$ To show \eqref{bound1}, we write
 \begin{align*}
 |\partial_t^{1-\frac1{2s}}F_s(t)|&\leq \int\frac{|F_s(r)-F_s(t)|}{|r-t|^{2-\frac1{2s}}}dr=\int_{|r|\leq|t|/2}\frac{|F_s(r)-F_s(t)|}{|r-t|^{2-\frac1{2s}}}dr\\&\quad\quad\quad+\int_{|t|/2<|r|\leq 2|t|}\frac{|F_s(r)-F_s(t)|}{|r-t|^{2-\frac1{2s}}}dr+\int_{|r|>2|t|}\frac{|F_s(r)-F_s(t)|}{|r-t|^{2-\frac1{2s}}}dr\\&=I_1+I_2+I_3.
 \end{align*}
 We begin with $I_1$. Notice that here $|r-t|\approx|t|$. Then,
$$I_1\lesssim\frac1{|t|^{2-\frac1{2s}}}\left(\int_{|r|\leq|t|/2}|F_s(r)|dr+\int_{|r|\leq|t|/2}|F_s(t)|dr\right)=I_{11}+I_{12}.$$
If $|t|\leq 1$, then 
\begin{align*}
I_{11}&\lesssim\frac1{|t|^{2-\frac1{2s}}}\int_0^{|t|}\frac{r}{(1+r^{\frac1s})^{\frac{N+2s}2}}dr\lesssim\frac{|t|}{|t|^{2-\frac1{2s}}}\int_0^{|t|}\frac{dr}{(1+r^{\frac1s})^{\frac{N+2s}2}}\\&\lesssim |t|^{\frac1{2s}}\leq 1.
\end{align*}
If $|t|>1$, then we can write
\begin{align*}
I_{11}&\lesssim\frac1{|t|^{2-\frac1{2s}}}\int_0^{1}\frac{r}{(1+r^{\frac1s})^{\frac{N+2s}2}}dr+\frac1{|t|^{2-\frac1{2s}}}\int_1^{|t|}\frac{r}{(1+r^{\frac1s})^{\frac{N+2s}2}}dr.
\end{align*}
Since $$\frac1{|t|^{2-\frac1{2s}}}\int_1^{|t|}\frac{r}{(1+r^{\frac1s})^{\frac{N+2s}2}}dr\leq\frac1{|t|^{2-\frac1{2s}}}\int_1^{|t|}r^{\frac{-N}{2s}}\approx \frac{1-|t|^{1-\frac{N}{2s}}}{|t|^{2-\frac1{2s}}}\lesssim \frac{|t|^{1-\frac1{2s}}}{|t|^{2-\frac1{2s}}}=\frac1{|t|},$$
then if $|t|>1$, $$I_{11}\lesssim \frac1{|t|^{2-\frac1{2s}}}+\frac1{|t|}\lesssim\frac1{|t|}.$$
Hence $\displaystyle I_{11}\lesssim \min\left(1,\frac1{|t|}\right).$ Since $F_s(t)=0$ for $t\leq 0$, to estimate $I_{12}$ we only need to consider positive $t$, therefore
\begin{align*}
I_{12}&=\frac{|F_s(t)|}{|t|^{1-\frac1{2s}}}\approx\frac{t^2}{t^{2-\frac1{2s}}(1+t^{\frac1s})^{\frac{N+2s}2}}=\frac{t^{\frac1{2s}}}{(1+t^{\frac1s})^{\frac{N+2s}2}}\lesssim\min\left(1,\frac1{t}\right).
\end{align*}
This finishes the estimate of $I_1$.  To deal with $I_2$, we distinguish two cases, according to whether $r$ has the same sign as $t$ or not. In the first case we write $r\in Y$, and in the second one, $r\in N$. In the case $r\in N$, with 
$|t|/2\leq|r|\leq 2|t|$, it turns out that $|r-t|\approx |t|$, and thus
\begin{align*}
I_{2,N} & :=\int_{r\in N,|t|/2\leq|r|\leq 2|t|}  \frac{|F_s(r) - F_s(t)|}{|r-t|^{2-\frac1{2s}}}\,dr\lesssim \frac1{|t|^{2-\frac1{2s}}}\int_0^{2|t|}\frac{r}{(1+r^{\frac1s})^{\frac{N+2s}2}}\,dr
+ \frac{|F_s(t)|}{|t|^{1-\frac1{2s}}}.
\end{align*}
Observe that this last expression is very similar to the ones in $I_{11}$ and $I_{12}$.
Then, by almost the same arguments we deduce that
$$I_{2,N}\lesssim \min\bigg(1,\frac1{|t|}\bigg).$$
To deal with the case when the sign of $s$ is the same as the one of $t$ (i.e., $s\in Y$),
we take into account that
$$|F_s(r)-F_s(t)|\leq \sup_{\xi\in[r,t]}|F_s'(\xi)|\,|r-t|.$$
Notice that \eqref{expression} tells us that if $|y|=1$, then $F_s(\xi)=P_s(y,\xi)$. Hence due to property (3) of this lemma 
and since in this case $|t|/2\leq|\xi|\leq 2|t|$, it is immediate to check that for this $\xi$ we  have
$$|F_s'(\xi)|\lesssim \frac1{(1+|t|^{\frac1s})^{\frac{N+2s}2}}.$$
Thus,
\begin{align*}
I_{2,Y} & :=\int_{r\in Y,|t|/2\leq|r|\leq 2|t|}  \frac{|F_s(r) - F_s(t)|}{|r-t|^{2-\frac1{2s}}}dr\\
&\lesssim \frac1{(1+|t|^{\frac1s})^{\frac{N+2s}2}}\int_{|r|\leq 2|t|} \frac{dr}{|r-t|^{1-\frac1{2s}}}\\&\lesssim
\frac{|t|^{\frac1{2s}}}{(1+|t|^{\frac1s})^{\frac{N+2s}2}}\lesssim \min\bigg(1,\frac1{|t|}\bigg).
\end{align*}
 
Finally we will deal with $I_3$. 
 \begin{align*}
 I_3\leq \int_{|r|>2|t|}\frac{|F_s(r)|}{|r-t|^{2-\frac1{2s}}}dr+\int_{|r|>2|t|}\frac{|F_s(t)|}{|r-t|^{2-\frac1{2s}}}dr=I_{31}+I_{32}.
 \end{align*}
 Notice that $|r-t|\approx|r|>2|t|$. If $|t|\leq 1$, then 
  \begin{align*}
  I_{31}\lesssim\int_{|t|}^1\frac{dr}{r^{1-\frac1{2s}}(1+r^{\frac1s})^{\frac{N+2s}2}}+\int_{1}^\infty\frac{dr}{r^{1-\frac1{2s}}(1+r^{\frac1s})^{\frac{N+2s}2}}.  \end{align*}
  Since $$\int_{|t|}^1\frac{dr}{r^{1-\frac1{2s}}(1+r^{\frac1s})^{\frac{N+2s}2}}\leq\int_{|t|}^1\frac{dr}{r^{1-\frac1{2s}}}\approx 1-|t|^{\frac1{2s}}\leq1$$ and $$\int_{1}^\infty\frac{dr}{r^{1-\frac1{2s}}(1+r^{\frac1s})^{\frac{N+2s}2}}\lesssim 1,$$ we have $\displaystyle I_{31}\lesssim1$ when $|t|\leq 1$.
 If $|t|>1$, then
  \begin{align*}
  I_{31}\lesssim\int_{|t|}^\infty\frac{dr}{r^{1-\frac1{2s}}(1+r^{\frac1s})^{\frac{N+2s}2}}\leq\int_{|t|}^\infty\frac{dr}{r^{2+\frac{N-1}{2s}}}\approx\frac1{|t|^{1+\frac{N-1}{2s}}}\lesssim\frac1{|t|}.
  \end{align*}
  Hence $\displaystyle I_{31}\lesssim \min\left(1,\frac1{|t|}\right)$. Finally
 \begin{align*}
 I_{32}\lesssim\frac{t}{(1+t^{\frac1s})^\frac{N+2s}2}\int_{t}^\infty\frac{dr}{r^{2-\frac1{2s}}}\lesssim\frac{t^{\frac1{2s}}}{(1+t^{\frac1s})^{\frac{N+2s}2}}\leq\min\left(1,\frac1{t}\right).
  \end{align*}
\end{proof}

\section{Localization estimates for $s=1/2$}\label{section-localization}

The main objective of this section is to show the following localization result.

\begin{theorem}\label{teoloc}
Let 
 $\nu$ be a distribution in $\R^{N+1}$ with 
$\|P*\nu\|_\infty \leq 1.$
Let $\vphi$ be a ${\mathcal C}^1$ function supported on a cube $Q\subset\R^{N+1}$ such that $\|\nabla\vphi\|_\infty\leq\ell(Q)^{-1}$. Then $\|P*(\vphi\nu)\|_\infty \lesssim 1.$
\end{theorem}


In the statement the operator $\nabla$ refers to $\nabla=(\nabla_x,\partial_t)$. Before proving Theorem \ref{teoloc} we need several lemmas and definitions.  
We say that a ${\mathcal C}^1$ function $\vphi$ is admissible for $Q\subset\R^{N+1}$ if it is supported on $Q\subset\R^{N+1}$ and satisfies 
\begin{equation}\label{admissiblephi}
\int_{Q}|\partial_t\vphi(x,t)|dxdt\lesssim \ell(Q)^N\quad\mbox{ and }\int_{\R^{N+1}}|(-\Delta_x)^{1/2}\vphi(x,t)|dxdt\lesssim\ell(Q)^N.
\end{equation}
Recall that $$(-\Delta_x)^{1/2}\vphi\approx\sum_{j=1}^NR_j\partial_j\vphi,$$ with $R_j$, $1\leq j\leq N$, being the Riesz transforms with Fourier multiplier $\xi_j/|\xi|$. Then setting $Q=Q_1\times I_Q$ with $Q_1\subset\R^N$ and $I_Q\subset\R$, one can write
\begin{align*}
\int_{\R^{N+1}}|(-\Delta_x)^{1/2}\vphi(x,t)|dxdt&\approx\int_{I_Q}\int_{\R^N}|\sum_{j=1}^NR_j\partial_j\vphi(x,t)|dxdt\\&\leq\sum_{j=1}^N\int_{I_Q}\int_{\R^N}|R_j\partial_j\vphi(x,t)|dxdt.\end{align*}
Therefore if $\vphi$ satisfies $$\int_{Q}|\partial_t\vphi(x,t)|dxdt\lesssim \ell(Q)^N\quad\mbox{ and }\quad\int_{I_Q}\|R_j\partial_j\vphi(\cdot,t)\|_{L^1(\R^N)}dt\lesssim\ell(Q)^N,\;\mbox{ for }1\leq j\leq N,$$ then $\vphi$ is admissible for $Q$.

\begin{lemma}\label{admissibility}
Let $\vphi$ be a ${\mathcal C}^1$ function supported on a cube $Q=Q_1\times I_Q\subset\R^{N+1}$, with $Q_1\subset\R^N$ and $I_Q\subset\R$, and such that 
\begin{enumerate}
\item $\displaystyle\|\partial_t\vphi\|_{L^1(Q)}\lesssim\ell(Q)^N$.
\item $\displaystyle\int_{I_Q}\|R_j\partial_j\vphi(\cdot,t)\|_{L^1(2Q_1)}dt\lesssim\ell(Q)^N$, for $j=1,2,\cdots,N$.
\end{enumerate}
Then $\vphi$ is admissible for $Q$.
\end{lemma}

\begin{proof}
Due to the above explanation, we only have to show that $$\int_{I_Q}\|R_j\partial_j\vphi(\cdot,t)\|_{L^1(\R^N\setminus2Q_1)}dt\lesssim\ell(Q)^N.$$
In fact,
\begin{align*}
\int_{I_Q}\|R_j\partial_j\vphi(\cdot,t)\|_{L^1(\R^N\setminus2Q_1)}dt&\approx\int_{I_Q}\int_{\R^N\setminus2Q_1}\left|\int_{Q_1}\partial_j\vphi(z,t)\frac{z_j-y_j}{|z-y|^{N+1}}dz\right|dydt\\&\lesssim\int_{I_Q}\int_{\R^N\setminus2Q_1}\int_{Q_1}\frac{|\vphi(z,t)|}{|z-y|^{N+1}}dzdydt\\&\lesssim\|\vphi\|_\infty\ell(I_Q)\ell(Q_1)^{N-1}\lesssim\ell(Q)^N,
\end{align*}
where the forelast inequality is obtained by integrating on annulus, for example. 
\end{proof}

\begin{coro}\label{standardtest}
A ${\mathcal C}^1$ function $\vphi$ supported on a cube $Q\subset\R^{N+1}$ with $\|\nabla\vphi\|_\infty\leq\ell(Q)^{-1}$ is admissible for $Q$.
\end{coro}
\begin{proof}
The first condition in lemma \ref{admissibility} is clearly fulfilled.  To show the second condition in lemma \ref{admissibility} we will use the Cauchy-Schwarz inequality and the $L^2-$boundedness of the Riesz transforms as follows:
\begin{align*}
\displaystyle\int_{I_Q}\|R_j\partial_j\vphi(\cdot,t)\|_{L^1(2Q_1)}dt&\lesssim \int_{I_Q}\ell(Q)^{N/2}\|R_j\partial_j\vphi(\cdot,t)\|_{L^2(2Q_1)}dt\\&\lesssim\int_{I_Q}\ell(Q)^{N/2}\|\partial_j\vphi(\cdot,t)\|_{L^2(2Q_1)}dt\lesssim \ell(Q)^N,
\end{align*}
Therefore, standard test functions supported on $Q$ are admissible for $Q$.
\end{proof}

It is worth mentioning here that if $\vphi$ is a standard ${\mathcal C}^2$ test function supported on $Q\subset\R^{N+1}$ satisfying $\|\partial_t\vphi\|_\infty\lesssim\ell(Q)^{-1}$ and $\|\Delta_x\vphi\|_\infty\lesssim\ell(Q)^{-2}$, then $\vphi$ is admissible for $Q$. The first condition in \eqref{admissiblephi} is clear and for the second one, notice that if $g=\Delta_x\varphi*_x k$, with $k(x)=\frac 1{|x|^{N-1}}$ and $*_x$ denoting the convolution on the $x$ variable, then taking the Fourier transform with respect to $x$, we get: $(-\Delta_x)^{1/2}\vphi=cg$, for a suitable constant $c\neq 0$. Then, integrating on annuli and using that $\|\Delta_x\vphi\|_\infty\lesssim\ell(Q)^{-2}$,
\begin{align}\label{standardtestok}
\int_{\R^{N+1}}|(-\Delta_x)^{1/2}\vphi(x,t)|dxdt&\approx\int_{\R^{N+1}}|\Delta_x\varphi*_x k(x,t)|dxdt\nonumber\\&\lesssim\int_{I_Q}\int_{\R^N\setminus2Q_1}\int_{Q_1}\frac{|\varphi(y,t)|}{|x-y|^{N+1}}dydxdt+\int_{I_Q}\int_{2Q_1}\int_{Q_1}\frac{|\Delta_y\varphi(y,t)|}{|x-y|^{N-1}}dydxdt\\&\lesssim\sum_{k=1}^\infty\frac{\ell(Q)^{N+1}(2^k\ell(Q))^{N}}{(2^k\ell(Q))^{N+1}}+\ell(Q)^N\lesssim\ell(Q)^N\nonumber.
\end{align}
\newline

The following lemma shows an $N-$growth condition that every distribution $\nu$ fulfilling the hypothesis of Theorem \ref{teoloc} satisfies.

\begin{lemma}\label{lemgrow}
Let $\nu$ be a distribution in $\R^{N+1}$ with
$\|P*\nu\|_\infty \leq 1.$
If $\vphi$ is a ${\mathcal C}^1$ function admissible for $Q\subset\R^{N+1}$, then
$|\langle \nu,\vphi\rangle|\lesssim \ell(Q)^N.$
\end{lemma}

\begin{proof}
Since $P$ is the fundamental solution of $\Theta^{1/2}$, we can write
$$|\langle \nu,\vphi \rangle|=|\langle \nu,\Theta^{1/2}\vphi*P \rangle|\leq|\langle P*\nu,(-\Delta_x)^{1/2}\vphi \rangle|+|\langle P*\nu,\partial_t\vphi \rangle|\lesssim\ell(Q)^N,$$
because $\|P*\nu\|_\infty \leq 1$ and $\vphi$ is admissible for $Q$, so it satisfies \eqref{admissiblephi}.
\end{proof}

We will say that a distribution $\nu$ has $N$-growth if for any $Q\subset\R^{N+1}$ and any $\vphi\in{\mathcal C}^1$ admissible for $Q$, $|\langle \nu,\vphi\rangle|\lesssim \ell(Q)^N$.

\begin{lemma}\label{punt}
Let $\nu$ be a distribution in $\R^{N+1}$ with $N$-growth. If $\vphi$ is a ${\mathcal C}^1$ function supported on a cube $Q\subset\R^{N+1}$ with $\|\nabla\varphi\|_\infty\leq \ell(Q)^{-1}$, then the distribution $\displaystyle P*\vphi\nu$ is a locally integrable function and there exists a point $\bar x_0\in\frac14Q$ such that \begin{equation}\label{pointineq}|(P*\vphi\nu)(\bar x_0)|\lesssim 1.\end{equation}
\end{lemma}

\begin{proof}
We will show that the mean of $\displaystyle f=P*\vphi\nu$ on $\frac14Q$ is bounded by a constant. Hence at many Lebesgue points of 
$f$ the inequality \eqref{pointineq} holds. 

We only need to show that $\displaystyle P*\vphi\nu$ is integrable on $2Q$. In fact, we will prove a stronger statement, namely that $\displaystyle P*\vphi\nu$  is in $L^p(2Q)$ for each $1< p<\frac{N+1}{N}$. Indeed, fix any $q$ satisfying $N+1<q<\infty$ and let $p$ be its dual exponent, so that $1< p<\frac{N+1}{N}$. We need to estimate the action of $\displaystyle P*\vphi\nu$ on test functions $\psi$ supported on $2Q$ in terms of $\|\psi\|_q$. We clearly have $$\langle P*\vphi\nu,\psi\rangle=\langle\nu,\vphi(P*\psi)\rangle.$$
We claim that the test function \begin{equation}\label{testfunction}h=\frac{\vphi(P*\psi)}{\ell(Q)^{(N+1)/p-N}\|\psi\|_q}\end{equation} satisfies the inequalities in lemma \ref{admissibility} and hence it is and admissible function. Once this is proved we get
$$|\langle P*\vphi\nu,\psi\rangle|=|\langle \nu,\vphi(P*\psi)\rangle|=\ell(Q)^{(N+1)/p-N}\|\psi\|_q|\langle\nu,h\rangle|\lesssim \ell(Q)^{(N+1)/p}\|\psi\|_q$$
and so $$\|P*\vphi\nu\|_{L^p(2Q)}\lesssim\ell(Q)^{(N+1)/p}.$$
Hence
\begin{align*}
\frac1{|\frac14Q|}\int_{\frac14Q}|(P*\vphi\nu)(x)|dx\leq&\;\frac{4^{N+1}}{|Q|}\int_Q|(P*\vphi\nu)(x)|dx\\\leq&\;4^{N+1}\left(\frac1{|Q|}\int_Q|(P*\vphi\nu)(x)|^pdx\right)^{1/p}\leq C
\end{align*}
which completes the proof of the lemma.

To prove the claim, we need to show that $h$ satisfies the inequalities in lemma \ref{admissibility}, that is
$$\|\nabla h\|_{L^1(2Q)}\lesssim\ell(Q)^N\quad\mbox{and}\quad\int_{I_Q}\|R_j(\partial_jh)\|_{L^1(2Q_1)}dt\lesssim\ell(Q)^N,\;1\leq j\leq N.$$ 
Or equivalently
\begin{equation}\label{equiv}
\|\nabla\left(\vphi(P*\psi)\right)\|_{L^1(Q)}\lesssim\ell(Q)^{\frac{N+1}p}\|\psi\|_q\quad\mbox{and}\quad\int_{I_Q}\|R_j\partial_j\big(\vphi(P*\psi)\big)\|_{L^1(2Q_1)}dt\lesssim\ell(Q)^{\frac{N+1}p}\|\psi\|_q,
\end{equation}
for $1\leq j\leq N.$ For the first inequality in \eqref{equiv} apply H\"older's inequality to obtain
\begin{align*}
\|\nabla\left(\vphi(P*\psi)\right)\|_{L^1(2Q)}&\lesssim\ell(Q)^{\frac{N+1}p}\|\nabla\left(\vphi(P*\psi)\right)\|_{L^q(2Q)}\\&\leq\ell(Q)^{\frac{N+1}p}\left(\|\nabla\vphi(P*\psi)\|_{L^q(2Q)}+\|\vphi\nabla(P*\psi)\|_{L^q(2Q)}\right)\\&=\ell(Q)^{\frac{N+1}p}(A_1+A_2).
\end{align*}
Since $\|\nabla\vphi\|_{\infty}\leq\ell(Q)^{-1}$ and by H\"older again,
\begin{align*}
A_1&\lesssim\ell(Q)^{-1}\left(\int_{2Q}\left(\int_{2Q}\frac{|\psi(\bar y)|}{|\bar y-\bar x|^N}d\bar y\right)^{q}d\bar x\right)^{1/q}\\&\leq\ell(Q)^{-1}\|\psi\|_{L^q(2Q)}\left(\int_{2Q}\left(\int_{2Q}\frac{d\bar y}{|\bar y-\bar x|^{Np}}\right)^{q/p}d\bar x\right)^{1/q}\\&\lesssim\ell(Q)^{-1}\|\psi\|_{L^q(2Q)}\left(\ell(Q)^{N+1}\ell(Q)^{(N+1-Np)\frac qp}\right)^{1/q}=\|\psi\|_{L^q(2Q)}.
\end{align*}
Applying \cite[Theorem 4.12]{duandikoetxea}, for example, we deduce that the singular integral with kernel $\partial_t P(x,t)$ is bounded in $L^q(\R^{N+1})$, $1<q<\infty$.
Therefore
$\displaystyle\|\vphi\,\partial_t(P*\psi)\|_{L^q(2Q)}\lesssim\|\psi\|_{L^q(2Q)}$. The estimate of the $L^q(Q)-$norm of $\vphi\partial_j(P*\psi)$  in $A_2$ is analogous. Finally $A_2\lesssim\|\psi\|_{L^q(2Q)}$ so the first inequality in \eqref{equiv} is proven.

For the second one write
$$\int_{2Q}|R_j\partial_j(\vphi(P*\psi))|\leq\int_{2Q}|R_j(\partial_j\vphi(P*\psi))|+\int_{2Q}|R_j(\vphi\partial_j(P*\psi))|=A_3+A_4.$$
Using H\"older, the $L^q(\R^N)-$boundedness of the Riesz transform and arguing similar to what we have just done for the term $A_1$,
\begin{align*}
A_3&\lesssim\ell(Q)^{\frac{N+1}p}\left(\int_{2Q}\big|\partial_j\vphi(\bar x)(P*\psi)(\bar x)\big|^qd\bar x\right)^{1/q}\leq\ell(Q)^{\frac{N+1}p-1}\left(\int_{2Q}\big|(P*\psi)(\bar x)\big|^qd\bar x\right)^{1/q}\\&\leq \ell(Q)^{\frac{N+1}p-1}\left(\int_{2Q}\left(\int_{2Q}\frac{|\psi(\bar y)|}{|\bar y-\bar x|^N}d\bar y\right)^{q}d\bar x\right)^{1/q}\leq\ell(Q)^{\frac{N+1}p}\|\psi\|_{L^q(2Q)}.
\end{align*}
Analogously,
\begin{align*}
A_4\lesssim\ell(Q)^{\frac{N+1}{p}}\left(\int_{2Q}\big|\vphi(\bar x)\partial_j(P*\psi)(\bar x)\big|^qd\bar x\right)^{1/q}\lesssim \ell(Q)^{\frac{N+1}p}\|\psi\|_{L^q(2Q)},
\end{align*}
by an argument similar to the one of $A_2$.
\end{proof}

{\em Proof of Theorem $3.1$.}
Take $\bx\in\frac32 Q$ and write $P_{\bar x}(\bar y)=P(\bar x-\bar y)$. We have to show that $\displaystyle\left|(P*\vphi\nu)(\bar x)\right|\lesssim 1$.  Write
\begin{align*}
\left|(P*\vphi\nu)(\bar x)\right|&\leq\left|(P*\vphi\nu)(\bar x)-\varphi(\bx)(P*\nu)(\bx)\right|+\|\vphi\|_\infty\|P*\nu\|_\infty\\&\leq\left|(P*\vphi\nu)(\bar x)-\varphi(\bx)(P*\nu)(\bx)\right|+C.
\end{align*}
Let $\psi$ be a ${\mathcal C}^1$ function such that $\psi\equiv 1$ in $2Q$, $\psi\equiv 0$ in $(4Q)^c$ and $\|\nabla\psi\|_\infty\leq\ell(Q)^{-1}$. We need to resort a standard regularization process. Take $\chi\in\mathcal{C}^{\infty}(B(0,1))$ such that $\int\chi(x)dx=1$ and set $\chi_\ve(x)=\ve^{-n}\chi(x/\ve)$ and $P^\ve=\chi_\ve*P$. We want to estimate \begin{equation}\label{truncation} \left|(P^\ve*\vphi\nu)(\bar x)-\varphi(\bx)(P^\ve*\nu)(\bx)\right|\end{equation} uniformly on $\chi$ and $\ve$. Since, as $\ve$ tends to zero, \eqref{truncation} tends to $\displaystyle\left|(P*\vphi\nu)(\bar x)-\varphi(\bx)(P*\nu)(\bx)\right|$ for almost all $\bar x\in\R^{N+1}$, this allows the transfer of uniform estimates.
\begin{align*}
\left|(P^\ve*\vphi\nu)(\bar x)-\varphi(\bx)(P^\ve*\nu)(\bx)\right|&=\left|\langle\nu,\vphi(\by)P^\ve_{\bx}(\by)-\vphi(\bx)P^\ve_{\bx}(\by)\rangle\right|\\&\leq\left|\langle\nu,\psi P^\ve_{\bx}\left(\vphi-\vphi(\bx)\right)\rangle\right|+\left|\langle\nu,(1-\psi)\vphi(\bx)P^\ve_{\bx}\rangle\right|\\&=A+B.
\end{align*}
To estimate term $A$, we will show that $h(\by)=\ell(Q)^N\psi(\by)P_{\bx}^\ve(\by)(\vphi(\by) -\vphi(\bx))$ is admissible for $4Q$ and then apply lemma \ref{lemgrow}. Since the function $$\phi(\by)=\psi(\by)\left(\vphi(\by) -\vphi(\bx)\right)$$ satisfies $\|\nabla \phi\|_\infty\lesssim\ell(Q)^{-1}$, for $\bx\neq\by$ we have $|P_{\bx}^\ve(\by)|\leq|\bx-\by|^{-N}$ and by the mean value theorem and lemma \ref{bound},$$|\phi(\by)\nabla P_{\bx}^\ve(\by)|\lesssim\|\nabla\vphi\|_\infty|\bx-\by|^{-N}\lesssim\ell(Q)^{-1}|\bx-\by|^{-N},$$  we have 
$$\|\nabla h\|_{L^1(4Q)}\lesssim\ell(Q)^N\ell(Q)^{-1}\int_{4Q}\frac{d\by}{|\bx-\by|^N}\lesssim\ell(Q)^N,$$ which is condition $(1)$ in lemma \ref{admissibility}.
To show $(2)$ in lemma \ref{admissibility} for the function $h$ we have to prove \begin{equation}\label{condition2}\int_{4I_Q}\|R_j\partial_jh(\cdot, t)\|_{L^1(8Q_1)}dt\lesssim\ell(Q)^N,\;\;j=1,2,\cdots, N.\end{equation}
Applying H\"older's inequality for some $q>1$ to be chosen later and using the $L^q(\R^N)$-boundedness of the Riesz transform $R_j$, we have
\begin{align*}
\int_{4I_Q}\|R_j\partial_jh(\cdot, t)\|_{L^1(8Q_1)}dt&=\int_{4I_Q}\int_{8Q_1}|R_j\partial_jh(x,t)|dxdt\\&\leq\int_{4I_Q}\ell(Q)^{N/p}\|R_j(\partial_j h)(\cdot, t)\|_{L^q(\R^N)}dt\\&\lesssim\ell(Q)^{N/p}\int_{4I_Q}\|\partial_j h(\cdot,t)\|_{L^q(\R^N)}dt.
\end{align*}
To estimate the last integral, write $$\|\partial_j h(\cdot,t)\|_{L^q(\R^N)}\lesssim\ell(Q)^N\left(\|\partial_j\phi P_{\bar x}^{\ve}\|_{L^q(\R^N)}+\|\phi\partial_jP_{\bar x}^{\ve}\|_{L^q(\R^N)}\right)=A_1+A_2.$$
Using that $\|\nabla \phi\|_\infty\lesssim\ell(Q)^{-1}$ and \begin{equation}\label{P}|P_{\bar x}^\ve(\bar y)|=|(\chi_\ve*P_{\bar x})(\bar y)|\lesssim\frac1{|\bar x-\bar y|^N}\lesssim\frac1{\sqrt{|t-s|}\;|x-y|^{N-1/2}},\end{equation} we get
\begin{align*}
\int_{4I_Q}A_1\;dt&=\ell(Q)^N\int_{4I_Q}\left(\int_{4Q_1}|\partial_j\phi(y,t)P_{\bar x}^\ve(\bar y)|^qdy\right)^{1/q}dt\\&\lesssim\ell(Q)^{N-1}\int_{4I_Q}\frac{dt}{\sqrt{|t-s|}}\left(\int_{4Q_1}\frac{dy}{|x-y|^{(N-1/2)q}}\right)^{1/q}\\&\lesssim\ell(Q)^{\frac Nq-\frac12}\int_{4I_Q}\frac{dt}{\sqrt{|t-s|}}\\&\lesssim\ell(Q)^{N/q},
\end{align*}
chosing $1<q<N/(N-1/2)$. Notice that $\displaystyle|(\chi_\ve*\partial_jP)(\bar y)|\lesssim\frac1{|\bar y|^{N+1}},$ for $\bar y\neq 0$. Therefore, using the mean value theorem and arguing as above we get
\begin{align*}
\int_{4I_Q}A_2\;dt&=\ell(Q)^N\int_{4I_Q}\|\phi\partial_jP_{\bar x}^\ve\|_{L^q(\R^N)}dt\\&=\ell(Q)^N\int_{4I_Q}\left(\int_{4Q_1}|\psi(\bar y)(\varphi(\bar y)-\varphi(\bar x))\partial_jP_{\bar x}^{\ve}(\bar y)|^qdy\right)^{1/q}dt\\&\lesssim\ell(Q)^{N-1}\int_{4I_Q}\frac{dt}{\sqrt{|t-s|}}\left(\int_{4Q_1}\frac{dy}{|x-y|^{(N-1/2)q}}\right)^{1/q}\\&\lesssim\ell(Q)^{\frac Nq-\frac12}\int_{4I_Q}\frac{dt}{\sqrt{|t-s|}}\\&\lesssim\ell(Q)^{N/q}.
\end{align*}
Since \eqref{condition2} holds, $h$ is admissible for $4Q$ and by lemma \ref{lemgrow} we get $A\lesssim 1$.\newline

To estimate term $\displaystyle B=\left|\langle\nu,(1-\psi)\vphi(\bx)P^\ve_{\bx}\rangle\right|$ we will use lemma \ref{punt}, i.e. the fact that there exists $\bar x_0\in Q$ such that 
$$
|(P*\psi\nu)(\bar x_0)|\lesssim 1.
$$
Since $\|P*\nu\|_\infty\leq 1$, we clearly have $\displaystyle |(P*(1-\psi)\nu)(\bar x_0)|\lesssim C$.  The analogous inequality holds as well for the regularized potentials appearing in $B$, uniformly in $\ve$, and therefore

$$
B\leq\|\varphi\|_\infty|\langle \nu,(1-\psi)P^\ve_{\bar x} \rangle|\lesssim |\langle \nu,(1-\psi)(P^{\ve}_{\bar x}-P^{\ve}_{\bar x_0})\rangle|+C.
$$
To estimate $\displaystyle |\langle \nu,(1-\psi)(P^{\ve}_{\bar x}-P^{\ve}_{\bar x_0})\rangle|$, we decompose $\R^{N+1}\setminus\{\bar x\}$ into a union of rings 
$$\{\bar z\in\R^{N+1}: 2^k\ell(Q)\leq|\bar z-\bar x|\leq 2^{k+1}\ell(Q)\},\;\;k\in\Z$$
and consider functions $\{\varphi_k\}$ in ${\mathcal C}^1(\R^{N+1})$  supported in  
$${\mathcal A}_k=\{\bar z\in\R^{N+1}: 2^{k-1}\ell(Q)\leq|\bar z-\bar x|\leq 2^{k+2}\ell(Q)\},\;\;k\in\Z$$ such that 
$\|\nabla\varphi_k\|_\infty\lesssim (2^k\ell(Q))^{-1}$ and such that $\sum_k\varphi_k=1$ in $\R^{N+1}\setminus\{\bar x\}$. Since $\bar x\in\frac32 Q$, the smallest ring ${\mathcal A}_k$ intersecting $(2Q)^c$ is ${\mathcal A}_{-3}$. Hence
\begin{align*}
|\langle \nu,(1-\psi)(P^{\ve}_{\bar x}-P^{\ve}_{\bar x_0})\rangle|&=\Big|\Big\langle \nu,\sum_{k=-3}^\infty\varphi_k(1-\psi)(P^{\ve}_{\bar x}-P^{\ve}_{\bar x_0})\Big\rangle\Big|\\&\leq\Big|\Big\langle \nu,\sum_{k\in I}\varphi_k(1-\psi)(P^{\ve}_{\bar x}-P^{\ve}_{\bar x_0})\Big\rangle\Big|+
\sum_{k\in J}\left|\left\langle \nu,\varphi_k(1-\psi)(P^{\ve}_{\bar x}-P^{\ve}_{\bar x_0})\right\rangle\right|,
\end{align*}
$I$ being the set of indices $k\geq-3$ such that supp$\;\varphi_k\cap 4Q\neq\emptyset$ and $J$ denoting the remaining indices ( i.e. $k\geq-3$ with $\varphi_k\equiv 0$ on $4Q$). Notice that the cardinality of $I$ is bounded by a dimensional constant. 

Set $$g=\ell(Q)^N\sum_{k\in I}\varphi_k(1-\psi)(P^{\ve}_{\bar x}-P^{\ve}_{\bar x_0}),$$ and for $k\in J$, $$g_k=2^k\big(2^k\ell(Q)\big)^N\varphi_k(P^{\ve}_{\bar x}-P^{\ve}_{\bar x_0}).$$
We will show now that we can apply lemma \ref{admissibility} to $g$ and $g_k,$ $k\in J$. Once this is available, lemma \ref{lemgrow} will give us
\begin{align*}
|\langle \nu,(1-\psi)(P^{\ve}_{\bar x}-P^{\ve}_{\bar x_0})\rangle|&\leq\ell(Q)^{-N}|\langle \nu,g\rangle|+\sum_{k\in J}2^{-k}\big(2^k\ell(Q)\big)^{-N}|\langle\nu,g_k\rangle|\lesssim 1+\sum_{k\in J}2^{-k}\lesssim 1.
\end{align*}
Notice first that the support of $g$ is contained in a cube $\lambda Q$ for some $\lambda$ depending only on $N$. On the other hand, the support of $g_k$ is contained in $2^{k+2}Q$. To apply lemma \ref{admissibility} we have to show that
\begin{equation}\label{g}
\|\nabla g\|_{L^1(\lambda Q)}\lesssim\ell(Q)^N,\hspace{.7cm}\int_{\lambda I_Q}\|R_j(\partial_jg)\|_{L^1(2\lambda Q_1)}dt\lesssim\ell(Q)^N,\;1\leq j\leq N,
\end{equation}
and for $k\in J$,
\begin{equation}\label{gj}
\|\nabla g_k\|_{L^1(2^{k+3} Q)}\lesssim\big(2^k\ell(Q)\big)^N,\hspace{.45cm}
\int_{2^{k+2}I_Q}\|R_j(\partial_jg_k)\|_{L^1(2^{k+3}Q_1)}dt\lesssim\big(2^k\ell(Q)\big)^N\;1\leq j\leq N,
\end{equation}
We check first \eqref{g}. Using lemma \ref{bound},
\begin{align*}
\|\partial_jg\|_{L^1(\lambda Q)}&\lesssim \ell(Q)^N\sum_{k\in I}\left(\frac1{\ell(Q)}\int_{\lambda Q\cap\supp(\vphi_k)}|P^{\ve}_{\bar x}(\bar y)-P^{\ve}_{\bar x_0}(\bar y)|d\bar y +\int_{\lambda Q\cap\supp(\vphi_k)}|\partial_j(P^{\ve}_{\bar x}(\bar y)-P^{\ve}_{\bar x_0}(\bar y))|d\bar y\right)\\&\lesssim\ell(Q)^N\left(\frac1{\ell(Q)}\int_{\lambda Q\cap\supp(\vphi_k)}\frac{d\bar y}{|\bar y-\bar x|^N}+\int_{\lambda Q\cap\supp(\vphi_k)}\frac{d\bar y}{|\bar y-\bar x|^{N+1}}\right)\lesssim\ell(Q)^N.
\end{align*}
The estimate for $\|\partial_tg\|_{L^1(\lambda Q)}$ is analogous. For the second inequality in \eqref{g}, let $1<q<\infty$ and $p$ be its dual exponent. Apply H\"older's inequality, the fact that Riesz transforms preserve $L^q(\R^N)$ and argue as in the estimates of the integrals of $A_1$ and $A_2$ (using \eqref{P}) to obtain:
\begin{align*}
\int_{\lambda I_Q}\|R_j\partial_jg(\cdot,t)\|_{L^1(2\lambda Q_1)}dt&\lesssim\int_{\lambda I_Q}\ell(Q)^{N/p}\|R_j\partial_jg(\cdot,t)\|_{L^q(2\lambda Q_1)}dt\\&\lesssim\int_{\lambda I_Q}\ell(Q)^{N/p}\|\partial_jg(\cdot,t)\|_{L^q(2\lambda Q_1)}dt\lesssim \ell(Q)^{N/p}\ell(Q)^{N/q}=\ell(Q)^N.
\end{align*}

To show the first inequality in \eqref{gj} we have to prove $\|\partial_tg_k\|_{L^1(2^{k+3} Q)}\lesssim\big(2^k\ell(Q)\big)^N$ and  $\|\partial_jg_k\|_{L^1(2^{k+3} Q)}\lesssim\big(2^k\ell(Q)\big)^N$, $1\leq j\leq N$, $k\in J$. For the $L^1-$norm of $\partial_jg_k$ we will use \begin{equation}\label{gradientestimatex}|\partial_jP^{\ve}_{\bar x}(\bar y)-\partial_jP^{\ve}_{\bar x_0}(\bar y)|\lesssim\frac{\ell(Q)}{(2^k\ell(Q))^{N+2}},\quad \bar y\in A_k,\;k\in J,\;1\leq j\leq N.\end{equation} 
Notice that \eqref{gradientestimatex} comes from a gradient estimate and lemma \ref{bound}. Hence $$|\partial_jg_k(\cdot,t)|\leq 2^k(2^k\ell(Q))^N\frac{\ell(Q)}{(2^k\ell(Q))^{N+2}}=\frac1{2^k\ell(Q)},$$ 
which is $\|\partial_jg_k\|_{L^1(2^{k+3} Q)}\lesssim\big(2^k\ell(Q)\big)^N$, $1\leq j\leq N$, $k\in J$.  

To show $\|\partial_tg_k\|_{L^1(2^{k+3} Q)}\lesssim\big(2^k\ell(Q)\big)^N$, we consider different cases. Write $\bar x=(x,t),\;\bar y=(y,u)$ and $\bar x_0=(x_0,t_0)$. If $t-u>0$ and $t_0-u>0$ a gradient estimate  together with lemma \ref{bound} gives us 
$$|\partial_tP^{\ve}_{\bar x}(\bar y)-\partial_tP^{\ve}_{\bar x_0}(\bar y)|\lesssim\frac{\ell(Q)}{(2^k\ell(Q))^{N+2}},\quad \bar y\in A_k,\;k\in J.$$
If  $t-u\leq 0$ and $t_0-u\leq 0$ then $\displaystyle|\partial_tP^{\ve}_{\bar x}(\bar y)-\partial_tP^{\ve}_{\bar x_0}(\bar y)|=0$ . If $t-u$ and $t_0-u$ have different signs, say $t-u>0$ and $t_0-u\leq 0$ for example, 
then for $ \bar y\in A_k,\;k\in J$, 
\begin{align*}|\partial_tP^{\ve}_{\bar x}(\bar y)-\partial_tP^{\ve}_{\bar x_0}(\bar y)|&=|\partial_tP^{\ve}_{\bar x}(\bar y)|\lesssim\frac1{|\bar x-\bar y|^{N+1}}+\frac{(t-u)^2}{|\bar x-\bar y|^{N+3}}\\&\lesssim \frac1{(2^k\ell(Q))^{N+1}}+\frac{\ell(Q)}{(2^k\ell(Q))^{N+2}}\lesssim \frac1{(2^k\ell(Q))^{N+1}}.\end{align*}
Notice that this last case only happens in a set of measure smaller or equal than $C(2^k\ell(Q))^N\ell(Q)$. Putting this estimates together we get
\begin{align*}
\|\partial_tg_k\|_{L^1(2^{k+3} Q)}&\lesssim 2^k(2^k\ell(Q))^N\left[\frac{\ell(Q)}{(2^k\ell(Q))^{N+2}}(2^k\ell(Q))^{N+1}+\frac{(2^k\ell(Q))^N\ell(Q)}{(2^k\ell(Q))^{N+1}}\right]\lesssim\big(2^k\ell(Q)\big)^N.
\end{align*} So the first inequality in \eqref{gj} holds. Moreover
\begin{align*}
\int_{2^{k+2}I_Q}\|R_j\partial_jg_k(\cdot,t)\|_{L^1(2^{k+3}Q_1)}dt&\lesssim\int_{2^{k+2}I}(2^{k+3}\ell(Q))^{N/p}\|R_j\partial_jg_k(\cdot,t)\|_{L^q(2^{k+3}Q_1)}dt\\&\lesssim\int_{2^{k+2}I_Q}(2^{k+3}\ell(Q))^{N/p}\|\partial_jg_k(\cdot,t)\|_{L^q(2^{k+3}Q_1)}dt\\&\lesssim (2^k\ell(Q))^{N/p}2^k\ell(Q)(2^k\ell(Q))^{N/q}\frac1{2^k\ell(Q)}\\&=(2^k\ell(Q))^N,
\end{align*}
which is the second inequality in \eqref{gj}.

\vv

Take $\bar x\in (\frac 3 2 Q)^c$ ($\bar x=(x,t)$). Then consider
$$\left|(P^\ve*\vphi\nu)(\bar x)\right|=\left|\langle\nu,\vphi P^\ve_{\bx}\rangle\right|=\ell(Q)^{-N}\left|\langle\nu,\ell(Q)^N\vphi P^\ve_{\bx}\rangle\right|\lesssim \ell(Q)^{-N}\ell(Q)^{N}=1$$
because the function $f(\bar y)=\ell(Q)^N\vphi(\by)P^\ve_{\bx}(\by)$ is supported on $Q$ and satisfies
$$\|\nabla f\|_\infty\leq\ell(Q)^{N}\left(\|\nabla\vphi\|_\infty\|P^\ve_{\bx}\|_\infty+\|\vphi\|_\infty\|\nabla P^\ve_{\bx}\|_\infty\right)\lesssim\ell(Q)^{N}\ell(Q)^{-N-1}=\ell(Q)^{-1},$$
which, by corollary \ref{standardtest}, implies that $f$ is admissible for $Q$ 
and we may apply lemma \ref{lemgrow} to obtain $\left|(P^\ve*\vphi\nu)(\bar x)\right|\leq C$ also in this case.

\vv

\qed

\vv

\section{Capacities and removable singularities}\label{section-capacities}

Given a compact set $E\subset \R^{N+1}$,
we define
\begin{equation}\label{eqsupgame1}
\gamma_\Theta^{1/2}(E) =\sup|\langle\nu,1\rangle|,
\end{equation}
where the supremum is taken over all distributions $\nu$ supported on $E$ such that
\begin{equation}\label{eqsupgame2}
\|P*\nu\|_{L^\infty(\R^{N+1})} \leq 1.
\end{equation}
We call $\gamma_\Theta^{1/2}(E)$ the $\frac12$-fractional caloric capacity of $E$.
We also define the capacity $\gamma_{\Theta,+}^{1/2}(E)$, 
in the same way as in \rf{eqsupgame1}, but with the supremum restricted to all positive measures $\nu$
supported on $E$ satisfying \rf{eqsupgame2}. 
Clearly,
$$\gamma_{\Theta,+}^{1/2}(E)\leq \gamma_{\Theta}^{1/2}(E).$$

Denote by the $\HH^{N}_{\infty}(E)$ the $N$-Hausdorff content of the compact set $E\subset\R^{N+1}$.

\vv

\begin{lemma}
For every Borel set $E\subset\R^{N+1}$,
$$\gamma_{\Theta,+}^{1/2}(E)\leq \gamma_\Theta^{1/2}(E) \lesssim \HH^{N}_{\infty}(E),$$
and
$$\dim_{H}(E)>N\quad\Longrightarrow \quad \gamma_\Theta^{1/2}(E)>0.$$
\end{lemma}


\begin{proof}
Just by the previous definitions, $\gamma_{\Theta,+}^{1/2}(E)\leq \gamma_\Theta^{1/2}(E)$. To prove $\gamma_\Theta^{1/2}(E) \lesssim \HH^{N}_{\infty}(E)$
 first notice that we can assume $E$ to be compact.
Let $\nu$
be a distribution supported on $E$ satisfying \eqref{eqsupgame2} 
and let $\{A_i\}_{i\in I}$ be a collection of sets in $\R^{N+1}$ which cover $E,$ and such that
$$\sum_{i\in I} \diam(A_i)^{N}\leq 2\,\HH^{N}_{\infty}(E).$$
For each $i\in I$, let $B_i$ be an open  ball centered in $A_i$ with $r(B_i)=\diam(A_i)$, so that $E\subset \bigcup_{i\in I} B_i$. By the compactness of $E$ we can assume $I$ to be finite.
By means  of the Harvey-Polking lemma \cite[Lemma 3.1]{Harvey-Polking}, we can construct ${\mathcal C}^\infty$ functions $\vphi_i$, $i\in I$, satisfying:
\begin{itemize}
\item $\supp\vphi_i\subset 2B_i$ for each $i \in I$,
\item $\|\nabla \vphi_i\|_\infty \lesssim 1/r(B_i)$, 
\item $\sum_{i\in I}\vphi_i = 1$ in $\bigcup_{i\in I} B_i$,
\end{itemize}
Hence, by corollary \ref{standardtest} and lemma \ref{lemgrow}
$$|\langle\nu,1\rangle| = \Big|\sum_{i\in I} \langle\nu,\vphi_i\rangle| \lesssim \sum_{i\in I}r(B_i)^{N}
= \sum_{i\in I}\diam(A_i)^{N}\lesssim \HH^{N}_{\infty}(E).$$
Since this holds for any distribution $\nu$ supported on $E$ satisfying \rf{eqsupgame2}, $\gamma_\Theta^{1/2}(E) \lesssim \HH^{N}_{\infty}(E)$.

\vv

To prove the second assertion in the lemma, let $E\subset\R^{N+1}$ be a Borel set satisfying 
$\dim_{H}(E)=s>N$. We may assume $E$ to be bounded and we can apply the well known Frostman lemma. 
Then, it follows that there exists some non-zero positive measure $\mu$ supported on $E$ satisfying
$\mu(B(\bar x,r))\leq r^s$ for all $\bar x\in\R^{N+1}$ and all $r>0$.

Thus, by Lemma \ref{bound}, we deduce that for all $\bar x\in\R^{N+1}$
$$|P* \mu(\bar x)|\lesssim \int \frac1{|\bar x-\bar y|^{N}} \,d\mu(\bar y)\lesssim \diam(E)^{s-N}
.$$
Therefore, 
$$\gamma_\Theta^{1/2}(E)\geq \frac{\mu(E)}{\|P*\mu\|_{L^\infty(\R^{N+1})}}>0.$$
\end{proof}

\vv
We say that a compact set $E\subset \R^{N+1}$ is removable for bounded $\frac 1 2$-fractional caloric functions (or $\frac12-$fractional caloric removable) if any bounded function $f:\R^{N+1}\to\R$ 
satisfying the $\frac12-$fractional heat equation in $\mathbb R^{N+1}\setminus E$, also satisfies the heat equation in the whole space $\R^{N+1}$. 


\vv
\begin{theorem}\label{teoremov}
A compact set $E\subset\R^{N+1}$ is  $\frac12-$fractional caloric removable if and only if $\gamma_\Theta^{1/2}(E)=0$.
\end{theorem}

\begin{proof}
It is clear that if $E$ is $\frac12-$fractional caloric removable, then $\gamma_\Theta^{1/2}(E)=0$. Conversely, 
suppose that $E\subset\R^{N+1}$ is not $\frac12-$fractional caloric removable. So there exists some function
$f:\R^{N+1}\to\R$ satisfying
$$\| f\|_{L^\infty(\R^{N+1})}<\infty,$$
and $\Theta^{1/2} (f)\equiv0$ in $\R^{N+1}\setminus E$ but $\Theta^{1/2}(f)\not\equiv0$ in $\R^{N+1}$ (in the distributional
sense). Since $E$ is compact there exists some (open) cube such that $E \subset Q$  and $\Theta ^{1/2} (f)\not\equiv0$ in $Q$. 
 Consider the distribution $\nu=\Theta^{1/2}(f)$. Since $\nu$ does not vanish
identically in $Q$, there exists some ${\mathcal C}^\infty$ function $\vphi$ supported on $Q$ such that $\langle \nu,\vphi\rangle >0$. Now take $g=P*(\vphi\nu)$. By theorem \ref{teoloc},
$$\| g\|_{L^{\infty}(\R^{N+1})} <\infty,$$
and thus, since $\supp(\vphi\nu)\subset Q\cap E=E$,
$$\gamma_\Theta^{1/2}(E) \geq  \frac{\langle \vphi\,\nu,1\rangle}{\|g\|_{L^{\infty}(\R^{N+1})}}= 
\frac{\langle \nu,\,\vphi\rangle}{\|g\|_{L^{\infty}(\R^{N+1})}}
>0.$$
\end{proof}
\vv

From the preceding lemmas, it is clear that, for any compact set $E\subset\R^{N+1}$,
\begin{itemize} 
\item if $\dim_{H}(E)>N$, then $E$ is not removable for bounded $\frac{1}{2}-$fractional caloric functions,
\item if $\HH^{N}(E)=0$ (in particular if $\dim_{H}(E)<N$), then 
$E$ is removable for bounded $\frac{1}{2}-$fractional caloric functions.
\end{itemize}
Thus the critical Hausdorff dimension for $\frac{1}{2}-$fractional caloric removability (and for $\gamma_\Theta$) is $N$.

We consider the operator
$$T\nu = P * \nu,$$
defined over distributions $\nu$ in $\R^{N+1}$. When $\mu$ is a finite measure, one can easily check that $T\mu(\bar x)$ is defined for
$\mu$-a.e.\ $\bar x\in\R^{N+1}$ by the integral 
$$T\mu(\bar x) = \int P(\bar x-\bar y)\,d\mu(\bar y).$$
For $\ve>0$, we also consider the truncated operator
$$T_\ve\mu(\bar x) = \int_{|\bar x-\bar y|>\ve} P(\bar x- \bar y)\,d\mu(\bar y),
$$
whenever the integral makes sense,
and for a function $f\in L^1_{loc}(\mu)$, we write
$$T_{\mu} f\equiv T (f\,\mu),\qquad T_{\mu,\ve} f\equiv T_\ve (f\,\mu).$$
We also denote
$$T_*\mu(x) = \sup_{\ve>0} |T_\ve \mu(x)|,\quad  T_{*,\mu}f(x) = \sup_{\ve>0} |T_\ve(f\, \mu)(x)|.$$
We say that $T_\mu$ is bounded in $L^2(\mu)$ if the operators $T_{\mu,\ve}$ are bounded in $L^2(\mu)$
uniformly on $\ve>0$.

Given $E\subset \R^{N+1}$,
we define the capacity 
\begin{equation}\label{eqsupgame1'}
\wt\gamma_{\Theta,+}^{1/2}(E) =\sup\mu(E),
\end{equation}
where the supremum is taken over all positive measures $\mu$ supported on $E$ such that
\begin{equation}\label{eqsupgame2'}
\|T\mu\|_{L^\infty(\R^{N+1})} \leq 1,\qquad \|T^*\mu\|_{L^\infty(\R^{N+1})} \leq 1.
\end{equation}
Here $T^*$ is dual of $T$. That is,
$$T^*\mu(\bar x) = \int P(\bar y-\bar x)\,d\mu(\bar y).$$
Notice that by definition, $$\wt\gamma^{1/2}_{\Theta,+}(E)\lesssim \gamma^{1/2}_{\Theta,+}(E) $$

In the next theorem we characterize $\wt\gamma^{1/2}_{\Theta,+}(E)$ in terms of the positive measures supported on $E$ such that $T_\mu$ is bounded in $L^2(\mu)$.

\begin{theorem}
For any set $E\subset\R^{N+1}$,
$$
\wt\gamma^{1/2}_{\Theta,+}(E) \approx\sup\{\mu(E):\supp\mu\in E,\,\|T_\mu\|_{L^2(\mu)\to L^2(\mu)}\leq1\},
$$
with the implicit constant in the above estimate independent of $E$.
\end{theorem}

\begin{proof}
Denote
\begin{align*}
S& =\sup\{\mu(E):\supp\mu\in E,\,\|T_\mu\|_{L^2(\mu)\to L^2(\mu)}\leq1\}.
\end{align*}


The arguments to show that $\wt\gamma^{1/2}_{\Theta,+}(E)\approx S$ are standard. Indeed, let $\mu$ be a positive measure supported on $E$ such that
$\wt\gamma_{\Theta,+}(E)\leq 2\mu(E)$, $\|T\mu\|_{L^\infty(\R^{N+1})}\leq1$ and $\|T^*\mu\|_{L^\infty(\R^{N+1})}\leq1$. By a Cotlar type inequality analogous to the one in \cite[Lemma 5.4]{Mattila-Paramonov}, say, one deduces that
$$\|T_\ve\mu\|_{L^\infty(\mu)} \lesssim 1,\qquad \|T_\ve^*\mu\|_{L^\infty(\mu)} \lesssim 1,$$
uniformly on $\ve>0$.

 To obtain the boundedness of the operator $T_{\mu}$ in $L^2(\mu)$ we will use a  suitable $T1$ theorem with respect to a measure which may be non-doubling (see for example \cite[Th 3.21]{Tolsa-llibre}). Since $T_{\mu}$ and $T^*_{\mu}$ are bounded in $L^{\infty},$  as a consequence of the result in \cite{Hytonen-Martikainen}, 
to apply the $T1$ theorem in our case it is enough to check that the weak boundedness property 
is satisfied for balls with thin boundaries, that is,
$<T_{\mu}\chi_B, \chi_B> \le C\mu(2B)$, if $B$ is a  ball with thin boundary.
A  ball of radius $r(B)$ is said to have thin boundary if
\begin{equation}\label{thin}
\mu\{\bar x:\mbox{dist}(\bar x,\partial B)\le t r(B)\}\le t\mu(2B).
\end{equation}
Let's consider a ${\mathcal C}^{\infty}$ function $\phi$ with compact support in $2B$ such that $\phi\equiv1$ on $B.$ Then 
$$
 |<T_{\mu}\chi_B,\chi_B>|\le \int_B|T_{\mu}\phi|d\mu+\int_B|T_{\mu}(\phi-\chi_B)|d\mu.
$$
Using Theorem \ref{teoloc} one can see that the first term in the right hand side is bounded by $C\mu(B)$. 
To get a bound of the second term we will use that B has a thin boundary. Using the boundedness of $K$, property (a) in Lemma \ref{bound}, 
$$
\int_B|T_{\mu}(\phi-\chi_B)|d\mu\lesssim \int_{2B\setminus B}\int_B\frac{d\mu(\bar y)}{|\bar x-\bar y|^{N}}d\mu(\bar x)
\le \sum_{j\ge 0}\int_{\{\bar x\notin B:\tiny\mbox{dist}(\bar x;\partial B)\sim 2^{-j}r(B)\}}\int_B\frac{d\mu(\bar y)}{|\bar x-\bar y|^{N}}d\mu(\bar x).
$$ 
Given $j$  and $x\notin B$ such that $\mbox{dist}(\bar x,\partial B)\sim 2^{-j}r(B),$ since $\mu$ is a measure with $N$-growth, one has 
$$
\int_B\frac{d\mu(\bar y)}{|\bar x-\bar y|^{N}}\lesssim\sum_{k=-1}^{j}\int_{|\bar x-\bar y|\sim 2^{-k}r(B)} 
\frac{d\mu(\bar y)}{|\bar x-\bar y|^{N}}\lesssim \sum_{k=-1}^{j}\frac{\mu(B(\bar x,2^{-k}r(B))}{(2^{-k}r(B))^{N}}\lesssim j+2. 
$$
Therefore, by \ref{thin}
$$
\int_B|T_{\mu}(\phi-\chi_B)|d\mu\lesssim 
\sum_{j\ge 1}(j+2)\mu\{\bar x: \mbox{dist}(\bar x,\partial B)\sim 2^{-j}r(B)\}\lesssim \sum_{j\ge 0}\frac{j+2}{2^j}\mu(2B) \lesssim \mu(2B).
$$
Consequently,  the weak boundedness property is satified and by the $T1-$theorem it
 follows that $T_\mu$ is bounded in $L^2(\mu)$, with
$\|T_\mu\|_{L^2(\mu)\to L^2(\mu)}\lesssim1.$ So we deduce that $S\gtrsim \wt\gamma_{\Theta,+}(E)$.

To prove the converse estimate, let $\mu$ be a positive measure supported on $E$ be such that $\|T_\mu\|_{L^2(\mu)\to L^2(\mu)}\leq1$
and $S\leq 2\mu(E)$. From the $L^2(\mu)$ boundedness of $T_\mu$, one deduces that $T$ and $T^*$ are bounded
from the space of finite signed measures $M(\R^{N+1})$ to $L^{1,\infty}(\mu)$. That is, there exists some constant $C>0$ such that for any measure
$\nu\in M(\R^{N+1})$, any $\ve>0$, and any $\lambda>0$,
$$\mu\big(\big\{\bar x\in\R^{N+1}:|T_\ve \nu(\bar x)|>\lambda \big\}\big) \leq C\,\frac{\|\nu\|}\lambda,$$
and the same holds replacing $T_\ve$ by $T^*_\ve$.
The proof of this fact is analogous to the one of Theorem 2.16 in \cite{Tolsa-llibre}. 
Then, by a well known dualization of these estimates (essentially due to Davie and {\O}ksendal) and an application of Cotlar's inequality, one deduces
that there exists some function $h:E\to [0,1]$ such that
$$\mu(E)\leq C\,\int h\,d\mu,\quad \; \|T(h\,\mu)\|_{L^\infty(\R^{N+1})}\leq 1,\quad \; \|T^*(h\,\mu)\|_{L^\infty(\R^{N+1})}\leq 1.$$
See Theorem 4.6 and Lemma 4.7 from \cite{Tolsa-llibre} for the analogous arguments in the case of analytic capacity and also Lemma 4.2 from \cite{Mattila-Paramonov} for the precise vectorial version of the dualization of the weak $(1,1)$ estimates required in our situation, for example. So we have
$$\wt\gamma^{1/2}_{\Theta,+}(E)\geq \int h\,d\mu\approx \mu(E)\approx S.$$
\end{proof}




\section{The existence of removable sets with positive Hausdorff measure $\HH^{N}$}\label{section-removablesets}

In this section we will construct a self-similar Cantor set $E \subset \R^{N+1}$ with positive and finite $\HH^{N}$ measure which is $\frac{1}{2}-$caloric removable. To obtain this result we need the following theorem, showing that certain distributions are actually measures.

\begin{theorem}\label{teomeas}
Let $E\subset\R^{N+1}$ be a compact set with $\HH^{N}(E)<\infty$
and let $\nu$ be a distribution supported on $E$ such that 
$$\|P*\nu\|_\infty \leq 1.$$
Then $\nu$ is a signed measure, absolutely continuous with respect to $\HH^{N+1}|_E$ and 
there exists a Borel function $f:E\to\R$ such that $\nu=f\,\HH^{N}|_E$ and
$\|f\|_{L^\infty(\HH^{N}|_E)}\lesssim 1$.
\end{theorem}

In what follows, we say that a distribution $\nu$ in $\R^{N+1}$ has $N-$dimensional growth, if for any cube $Q\subset\R^{N+1}$
and any ${\mathcal C}^1$ function $\vphi$ admissible for $Q$, 
$\displaystyle|\langle \nu,\vphi\rangle|\leq \ell(Q)^{N}.$

Theorem \ref{teomeas} is a consequence of Lemma \ref{lemgrow} and the following result.

\begin{lemma}
Let $E\subset\R^{N+1}$ be a compact set with $\HH^{N}(E)<\infty$.
Let $\nu$ be a distribution supported on $E$ wich has $N$-dimensional growth. Then $\nu$ is a signed measure, absolutely continuous with respect to $\HH^{N}|_E$ and 
there exists a Borel function $f:E\to\R$ such that $\nu=f\,\HH^{N}|_E$ and 
$\|f\|_{L^\infty(\HH^{N}|_E)}\lesssim 1$.
\end{lemma}

\begin{proof}
First we will show that $\nu$ is a signed measure. By the Riesz representation theorem, it is enough
to show that, for any ${\mathcal C}^\infty$ function $\psi$ with compact support,
\begin{equation}\label{eq001}
|\langle \nu,\psi\rangle|\leq C(E)\,\|\psi\|_\infty,
\end{equation}
where $C(E)$ is some constant depending on $E$.

To prove \rf{eq001}, we fix $\ve>0$ and we consider a family of open cubes $Q_i$, 
$i\in I_\ve$, such that
\begin{itemize}
\item $E\subset \bigcup_{i\in I_\ve} Q_i$,
\item $\ell(Q_i)\leq \ve$ for all $i\in I_\ve$, and
\item $\sum_{i\in I_\ve} \ell(Q_i)^{N} \leq C\,\HH^{N}(E) + \ve$.
\end{itemize}
Since $E$ is compact, we can assume that $I_\ve$ is finite. 
By standard arguments, we can find a family of non-negative functions $\vphi_i$, $i\in I_\ve$, such that
\begin{itemize}
\item each $\vphi_i$ is supported on $2Q_i$ and $c\vphi_i$ is admissible for $2Q_i$, for some
absolute constant $c>0$, 
\item $\sum_{i\in I_\ve} \vphi_i = 1$ on $\bigcup_{i\in I_\ve} Q_i$, and in particular on $E$.
\end{itemize}
Indeed, to construct the family of functions $\vphi_i$ we can cover each cube $Q_i$ by a bounded number (depending on $n$)  
dyadic cubes $R_i^1,\cdots,R_i^m$ with side length $\ell(R_i^j)\leq \ell(Q_i)/8$ and then apply
the usual Harvey-Polking lemma (\cite[Lemma 3.1]{Harvey-Polking} to the family of cubes $\{R_{i,j}\}$.

We write
$$|\langle \nu,\psi\rangle|\leq \sum_{i\in I_\ve} |\langle \nu,\vphi_i\psi\rangle|.$$
For each $i\in I_\ve$, consider the function
$$\eta_i= \frac{\vphi_i\,\psi}{\|\psi\|_\infty + \ell(Q_i)\,\|\nabla \psi\|_\infty },$$
where $\nabla=(\nabla_x, \partial_t).$
We claim that $c\,\eta_i$ is admissible for $2Q_i$, for some absolute constant $c>0$. To check this,
just note that $\vphi_i\,\psi$ is supported on $2Q_i$ and satisfies
$$\|\nabla (\vphi_i\,\psi)\|_\infty \leq \|\nabla \vphi_i\|_\infty\,\|\psi\|_\infty
+ \| \vphi_i\|_\infty\,\|\nabla\psi\|_\infty \lesssim \frac1{\ell(Q_i)}\,\|\psi\|_\infty
+ \|\nabla\psi\|_\infty.$$
Hence, 
$$\|\nabla \eta_i\|_\infty \lesssim \frac1{\ell(Q_i)}.$$
So, by corollary \ref{standardtest}, the claim above holds and, consequently, by the assumptions in the lemma,
$$|\langle \nu,\eta_i\rangle|\lesssim \ell(Q_i)^{N}.$$

From the preceding estimate, we deduce that
$$|\langle \nu,\psi\rangle|\leq \sum_{i\in I_\ve} |\langle \nu,\vphi_i\psi\rangle|
\lesssim \sum_{i\in I_\ve} \ell(Q_i)^{N} \,\big(
\|\psi\|_\infty + \ell(Q_i)\,\|\nabla \psi\|_\infty \big).$$
Since $\ell(Q_i)\leq\ve$ for each $i$, we obtain that 
\begin{align*}
|\langle \nu,\psi\rangle| &
\lesssim \sum_{i\in I_\ve} \ell(Q_i)^{N} \,\big(
\|\psi\|_\infty + \ve\,\|\nabla \psi\|_\infty \big)\\
&\lesssim \big(\HH^{N}(E) + \ve\big)\,
\big(
\|\psi\|_\infty + \ve\,\|\nabla \psi\|_\infty \big).
\end{align*}
Letting $\ve\to0$, we get
$$|\langle \nu,\psi\rangle|\lesssim \HH^{N}(E)\,\|\psi\|_\infty,$$
which gives \rf{eq001} and proves that $\nu$ is a finite signed measure, as wished.

It remains to show that
there exists some Borel function $f:E\to\R$ such that $\nu=f\,\HH^{N}|_E$, with 
$\|f\|_{L^\infty(\HH^{N}|_E)}\lesssim1$.
To this end, let $g$ be the density of $\nu$ with respect to its variation $|\nu|$, so that
$\nu = g\,|\nu|$ with $g(\bar x)=\pm1$ for $|\nu|$-a.e.\ $\bar x\in\R^{N+1}$.
We will show that
\begin{equation}\label{eqdens92}
\limsup_{r\to0}\frac{|\nu|(B(\bar x,r))}{r^{N}}\lesssim1 \quad\mbox{ for $|\nu|$-a.e.\ $\bar x\in\R^{N+1}$}.
\end{equation}
This implies that $|\nu|=\wt f\,\HH^{N}|_E$ for some non-negative function $\wt f\lesssim1$. This fact is well known (see Theorem 6.9 from \cite{Mattila-gmt}).  

So to complete the proof of the lemma it suffices to show \rf{eqdens92} (since then we will have
$\nu= g\,\wt f\,\HH^{N}|_E$ with $|g\wt f|\lesssim1$). Notice that, by the Lebesgue differentiation theorem, 
$$\lim_{r\to 0}\frac1{|\nu|(B(\bar x,r))} \int_{B(\bar x,r)}|g(\bar y)-g(\bar x)|\,d|\nu|(\bar y) =0
\quad\mbox{ for $|\nu|$-a.e.\ $\bar x\in\R^{N+1}$.}$$
Let $\bar x\in E$ be a Lebesgue point for $|\nu|$ with $|g(\bar x)|=1$,  let $\ve>0$ to be chosen below,
 and let $r_0>0$ be small enough such that, for $0<r\leq r_0$,
$$\frac1{|\nu|(B(\bar x,r))} \int_{B(\bar x,r)}|g(\bar y)-g(\bar x)|\,d|\nu|(\bar y)<\ve.$$
Suppose first that
\begin{equation}\label{eqdoub*}
|\nu|(B(\bar x,2r))\leq 2^{N+3} \,|\nu|(B(\bar x,r)),
\end{equation}
and let $\vphi_{\bar x,r}$ be some non-negative ${\mathcal C}^\infty$ function supported on $B(\bar x,2r)$ which equals $1$ on 
$B(\bar x,r)$ such that $c\,\vphi_{\bar x,r}$ is admissible for the smallest  cube $Q$ containing $B(\bar x,2r)$,
so that
$$\left| \int \vphi_{\bar x,r}\,d\nu\right|\lesssim r^{N}.$$
Now observe that
\begin{align*}
\left|\int \vphi_{\bar x,r}\,d\nu - g(\bar x)\!\int \vphi_{\bar x,r}\,d|\nu| \right|
& = \left|\int \vphi_{\bar x,r}(\bar y) (g(\bar y)  - g(\bar x))\,d|\nu|(\bar y) \right| \lesssim \int_{B(\bar x,2r)} |g(\bar y)  - g(\bar x)|\,d|\nu|(\bar y) \\
& \leq \ve\,|\nu|(B(\bar x,2r))\leq \ve\,2^{N+3}\,|\nu|(B(\bar x,r)) \lesssim \ve \int \vphi_{\bar x,r}\,d|\nu|.
\end{align*}
Thus, if $\ve$ is chosen small enough, we deduce that
$$ \int \vphi_{\bar x,r}\,d|\nu| = |g(\bar x)|\int \vphi_{\bar x,r}\,d|\nu| \leq 2
\left|\int \vphi_{\bar x,r}\,d\nu\right|\lesssim r^{N}.$$
Therefore, using again that $\vphi_{\bar x,r}=1$ on $B(\bar x,r)$, we get
\begin{equation}\label{eqgro956}
|\nu|(B(\bar x,r))\lesssim r^{N}.
\end{equation}

To get rid of the doubling assumption \rf{eqdoub*}, notice that for $|\nu|$-a.e.\ $\bar x\in\R^{N+1}$ there
exists a sequence of balls $B(\bar x,r_k)$, with $r_k\to0$, satisfying \rf{eqdoub*} (we say that the balls
$B(\bar x,r_k)$ are $|\nu|$-doubling). Further, we may assume
that $r_k=2^{h_k}$, for some $h_k\in\N$. The proof of this fact is analogous to the one of Lemma 2.8
in \cite{Tolsa-llibre}. So for such a point $\bar x$, by the arguments above, we know that there exists
some $k_0>0$ such that
$$|\nu|(B(\bar x,r_k))\lesssim r_k^{N}\quad\mbox{ for $k\geq k_0$,}$$
assuming also that $\bar x$ is a $|\nu|$-Lebesgue point for the density $g$. Given an arbitrary $r\in (0,r_{k_0})$,
let $j$ be the smallest integer $r\leq 2^j$, and let $2^k$ be the smallest $j\leq k$ such that
the ball $B(\bar x,2^k)$ is $|\nu|$-doubling (i.e., \rf{eqdoub*} holds for this ball). Observe that
$2^k\leq r_{k_0}$.
Then, taking into account that the balls $B(\bar x,2^h)$ are non-doubling for $j\leq h<k$ and applying \rf{eqgro956} for
$r=2^k$, we obtain
$$|\nu|(B(\bar x,r))\leq |\nu|(B(\bar x,2^j)) \leq 2^{(n+3)(j-k)}\,|\nu|(B(\bar x,2^k)) \lesssim 2^{(n+3)(j-k)}\,2^{kn}
\leq 2^{jn}\approx r^{N}.$$
Hence, \rf{eqdens92} holds and we are done.
\end{proof}

\vv

Next we will construct a self-similar Cantor set $E\subset \R^{N+1}$ with positive and finite measure $\HH^N$ which is removable for the $\frac{1}{2}-$fractional heat equation.  Our example is inspired by the typical planar $1/4$ Cantor
set in the setting of analytic capacity (see \cite{Garnett} or \cite[p.\ 35]{Tolsa-llibre}, for example).

We construct the Cantor set $E$ as follows:
we let $E_0$ be the unit cube, i.e. $E_0=Q^0=[0,1]^{N+1}$. Next we replace $Q^0$ by $2^{N+1}$ disjoint closed cubes $Q^1_i$ with side length $2^{\frac{-(N+1)}{N}}$ such that  they are all contained in $Q^0$ and each one contains a vertex of $Q^0$. 

We proceed inductively:
In each generation $k$, we replace each  cube $Q_j^{k-1}$ of the previous generation by $2^{N+1}$  cubes $Q_i^{k}$ with side length $2^{(-k)\frac{N+1}{N}}$ which are contained in $Q_j^{k-1}$ and located in the same relative position to $Q_j^{k-1}$ as the cubes $Q_1^1,\ldots,Q^1_{2^{N+1}}$ with respect to $Q_0$. 

Notice that in each generation $k$ there are $2^{k(N+1)}$ closed  cubes with side length $2^{\frac{-(N+1)}{N}k}$. 
We denote by $E_k$ the union of all these  cubes from the $k$-th generation. By construction, $E_k\subset E_{k-1}$.
We let 
\begin{equation}\label{defE}
E=\bigcap_{k=0}^\infty E_k.
\end{equation}
It is easy to check that 
$\dist(Q_i^k,Q_h^k)\ge C(N) 2^{\frac{-(N+1)}{N}k}$ for $i\neq h$, and if $Q_i^k$, $Q_h^k$ are contained in the same  cube $Q^{k-1}_j$, then $\dist(Q_i^k,Q_h^k)\lesssim 2^{\frac{-(N+1)}{N}k}$.
Taking into account that, for each $k\geq0$,
$$\sum_{i=1}^{2^{k(N+1)}} \ell(Q_i^k)^{N} = 2^{k(N+1)}\cdot 2^{-k(N+1)}=1,$$
by standard arguments it follows that
$$0<\HH^N(E)<\infty.$$
Further, $\HH^N|_E$ coincides, modulo a constant factor, with the probability measure $\mu$ supported on $E$ which gives the same measure to all the cubes $Q_i^k$ of the same generation $k$, that is $\mu(Q_i^k)=2^{-k(N+1)}$.

\vv

\begin{theorem}
The Cantor set $E$ defined in \rf{defE} is removable for $\frac{1}{2}-$ fractional caloric functions.
\end{theorem}

\begin{proof}
We will suppose that $E$ is not removable and we will reach a contradiction. 
By Theorem \ref{teoremov}, there exists a distribution $\nu$ supported on $E$ such that 
$|\langle\nu,1\rangle|>0$ and
$$
\|P*\nu\|_{L^\infty(\R^{N+1})} \leq 1.
$$
By Theorem \ref{teomeas}, $\nu$ is a signed measure of the form
$$\nu=f\,\mu,\qquad \|f\|_{L^\infty(\mu)}\lesssim 1,$$
where $\mu$ is the probability measure supported on $E$ such that
$\mu(Q_i^k)=2^{-k(N+1)}$ for all $i,k$.
It is easy to check that $\mu$ (and thus $|\nu|$) has upper growth of degree $N$.  Recall that $T$ denotes the operator $T\nu=P*\nu$, defined over distributions $\nu$ in $\R^{N+1}$. Then, 
arguing as in \cite[Lemma 5.4]{Mattila-Paramonov}, it follows that there exists some constant $K$ such that
\begin{equation}\label{eqmatpar}
T_*\nu(\bar x)\leq K\,\quad \mbox{ for all $\bar x\in\R^{N+1}$.}
\end{equation}

For $\bar x\in E$, let $Q_{\bar x}^k$ the cube $Q_i^k$ that contains $\bar x$. Then we consider the auxiliary operator
$$\wt T_*\nu(\bar x) = \sup_{k\geq0} |T (\chi_{\R^{N+1}\setminus Q_{\bar x}^k}\nu)(\bar x)|.$$
By the separation condition between the cubes $Q_i^k$, the upper growth of $|\nu|$, and the condition \rf{eqmatpar}, it follows easily that
\begin{equation}\label{eqmatpar'}
\wt T_*\nu(\bar x)\leq K'\,\quad \mbox{ for all $\bar x\in E$,}
\end{equation}
for some fixed constant $K'$.

We will contradict the last estimate. To this end, consider a Lebesgue point $\bar x_0\in E$ (with respect to $\mu$) of the density $f=\frac{d\nu}{d\mu}$ such that $f(\bar x_0)>0$. The existence of this point is
guarantied by the fact that $\nu(E)>0$.
Given $\ve>0$ small enough to be chosen below, consider a cube $Q_i^k$ containing $\bar x_0$ such that
$$\frac1{\mu(Q_i^k)} \int_{Q_i^k} |f(\bar y)- f(\bar x_0)|\,d\mu(y)\leq \ve.$$
Given $m\gg1$, to be fixed below too, it is easy to check that if $\ve$ is chosen small enough (depending on $m$ and on $f(\bar x_0)$), then
the above condition ensures that
every cube $Q_j^h$ contained in $Q_i^k$ such that $k\leq h\leq k+m$ satisfies
\begin{equation}\label{eqmatp*}
\frac12 f(\bar x_0)\,\mu(Q^h_j)\leq
\nu(Q^h_j)\leq 2 f(\bar x_0)\,\mu(Q^h_j).
\end{equation}
Notice also that, writing $\nu=\nu^+-\nu^-$, since $f(\bar x_0)>0$,
$$\nu^-(Q_i^k) = \int_{Q_i^k} f^-(\bar y)\,d\mu(\bar y) \leq \int_{Q_i^k} |f(\bar y)-f(\bar x_0)|\,d\mu(\bar y) \leq \ve\,\mu(Q_i^k).$$

Let $\bar z= (z,u)$ be one of the upper corners of $Q^k_i$ (i.e.,   $u$ maximal in $Q^k_i$). Since
$|T(\chi_{Q^k_{\bar z}\setminus Q^{k+m}_{\bar z}}\nu(\bar z)| \leq 2\,\wt T_*\nu(\bar z)$, we have
$$\wt T_*\nu(\bar z) \geq \frac12\,|T(\chi_{Q^k_{\bar z}\setminus Q^{k+m}_{\bar z}}\nu)(\bar z)| \geq 
\frac12\,|T(\chi_{Q^k_{\bar z}\setminus Q^{k+m}_{\bar z}}\nu^+)(\bar z)| - \frac12\,|T(\chi_{Q^k_{\bar z}\setminus Q^{k+m}_{\bar z}}\nu^-)(\bar z)|.
$$
Using the fact that $\dist(\bar z,\,Q^k_{\bar z}\setminus Q^{k+m}_{\bar z}) \gtrsim \ell(Q^{k+m}_{\bar z})$, we get
$$|T(\chi_{Q^k_{\bar z}\setminus Q^{k+m}_{\bar z}}\nu^-)(\bar z)|\lesssim \frac{\nu^-(Q^k_i)}{\ell(Q^{k+m}_{\bar z})^N} 
\leq \ve\,\frac{\mu(Q^k_i)}{\ell(Q^{k+m}_{\bar z})^N} = \ve\,\frac{\mu(Q^k_i)}{2^{-m(N+1)}\ell(Q^{k}_i)^N} 
\lesssim 2^{m(N+1)}\,\ve.$$

To estimate $|T(\chi_{Q^k_{\bar z}\setminus Q^{k+m}_{\bar z}}\nu^+)(\bar z)|$ from below, recall  
that our kernel is 
$$P(\bar x)=\cfrac{t}{(t^2+|x|^2)^{\frac{N+1}2}}  \,\chi_{\{t>0\}}.$$
Then, by the choice of $\bar z$, it follows that
\begin{equation}\label{eqpos1}
P(\bar z- \bar y) \geq 0\quad \mbox{ for all $\bar y\in Q^k_{\bar z}\setminus Q^{k+m}_{\bar z}$.}
\end{equation}
We write
$$|T(\chi_{Q^k_{\bar z}\setminus Q^{k+m}_{\bar z}}\nu^+)(\bar z)| \geq \int_{Q^k_{\bar z}\setminus Q^{k+m}_{\bar z}} P(\bar z- \bar y)\,
d\nu^+(\bar y) = \sum_{h=k}^{k+m-1} \int_{Q^h_{\bar z}\setminus Q^{h+1}_{\bar z}} P(\bar z- \bar y)\,
d\nu^+(\bar y).$$
Taking into account \rf{eqpos1} and the fact that, for $k\leq h\leq k+m-1$, $Q^h_{\bar z}\setminus Q^{h+1}_{\bar z}$ contains a cube $Q^{h+1}_j$ such that for all $\bar y=(y,s)$,
$$0<u-s \approx |\bar y-\bar z|\approx\ell(Q^{h+1}_j),$$
using also \rf{eqmatp*},
we deduce
$$\int_{Q^h_{\bar z}\setminus Q^{h+1}_{\bar z}} P(\bar z- \bar y)\,
d\nu^+(\bar y) \gtrsim \frac{\nu^+(Q^{h+1}_j)}{\ell(Q^{h+1}_j)^N}\gtrsim f(\bar x_0)\,\frac{\mu(Q^{h+1}_j)}{\ell(Q^{h+1}_j)^N}
= f(\bar x_0),$$
Thus,
$$|T(\chi_{Q^k_{\bar z}\setminus Q^{k+m}_{\bar z}}\nu^+)(\bar z)| \gtrsim (m-1)\,f(\bar x_0).$$
Together with the previous estimate for $|T(\chi_{Q^k_{\bar z}\setminus Q^{k+m}_{\bar z}}\nu^-)(\bar z)|$, this tells us that
$$\wt T_*\nu(\bar z) \gtrsim (m-1)\,f(x_0) - C\,2^{m\frac{N+1}{N}}\,\ve,$$
for some fixed $C>0$.
It is clear that if we choose $m$ big enough and then $\ve$ small enough, depending on $m$, this
lower bound contradicts \rf{eqmatpar'}, as wished.
\end{proof}

\vv

\section{Neither analytic capacity nor Newtonian capacity are comparable to $\gamma_\Theta^{1/2}$.}\label{section-comparability}
In this section we will show that, in the plane, the capacity $\gamma_\Theta^{1/2}$ is not comparable to analytic capacity nor to Newtonian capacity; two classical capacities related to complex analysis and potential theory with critical dimension $1$, as $\gamma_\Theta^{1/2}$. More precisely, we will construct two sets, the first one will be a compact set $E_1\subset\C$ with $\gamma_\Theta^{1/2}(E_1)>0$ but vanishing Newtonian capacity. The second example we will be a compact set $E_2\subset\C$  with positive analytic capacity but removable for the $\frac{1}{2}-$fractional heat equation. 

First we recall some definitions. For a compact set $E\subset\mathbb C$, the analytic capacity of $E$ is defined as 
$$
\gamma(E)=\sup\{|f'(\infty)|\},
$$
where the supremum is taken over all analytic functions $f:\mathbb C\setminus E \rightarrow \mathbb C$ with $|f|\le 1$ in $\mathbb C\setminus E$ and 
$
\displaystyle{
f'(\infty)=\lim_{z\rightarrow \infty} z(f(z)-f(\infty))}.
$

The Newtonian capacity of a compact set $E\subset \mathbb C$ can be defined as 
$$
C(E)=\sup\{\mu(E):\mbox{spt}(\mu)\subset E, \, \int_{\mathbb C}\frac{1}{|z-w|}d\mu(w)\le 1 \, \, \text{ for } z\in \mathbb C\}.
$$

It is well-known that if $E\subset\C$ is a segment of length $\ell$, then $\gamma (E)=\ell/4$ and $C(E)=0$ (see \cite[Proposition 1.4]{Tolsa-llibre} and \cite[Corollary 3.6]{Garnett} respectively). 

We will prove that if $E_1\subset\C$ is a horizontal segment then $\gamma_\Theta^{1/2}(E_1)>0$. On the other hand, if $E_2\subset\C$ is a vertical segment, $\gamma_\Theta^{1/2}(E_2)=0$. Consequently $C(E_1)$ is not comparable to $\gamma_\Theta^{1/2}(E_1)$ and $\gamma_\Theta^{1/2}(E_2)$ is not comparable to $\gamma(E_2)$. The result reads as follows.

\begin{propo}\label{a} Let $L$ be a positive number. Set $E_1=\{(x,0)\subset\R^2: 0\le \, x\le L\}$ and $E_2=\{(0,t)\subset\R^2: 0\le t\le L\}.$ Then\newline

\begin{enumerate} 

 \item[a)]  $\gamma_\Theta^{1/2}(E_1)>0.$\newline
 
 \item[b)] $\gamma_\Theta^{1/2}(E_2)=0.$
\end{enumerate}
 \end{propo}\vv
 
\begin{proof}\begin{enumerate}
 \item[a)]
 Let $\chi_{E_1}$ be the characteristic function of $E_1$ and consider the positive measure $$\mu=\chi_{E_1}dx$$ and the function 
 $$
 f(a,s)=\big(P\ast\chi_{E_1}\big)(a,s), 
 $$
 where $P$ is the fundamental solution of the $\frac 12$-fractional heat equation in $\R^2$, that is $\displaystyle P(x,t)=\frac{t}{x^2+t^2}\chi_{t>0}$ .
 Clearly $f(a,s)=0$ if $s\le 0.$ For $s>0$ we have 
 
\begin{align*}
f(a,s)=\int_0^L\frac{s}{(a-x)^2+(s)^2}dx=\arctan(\frac{L-a}{s})-\arctan(\frac{-a}{s})\le \pi, 
\end{align*}
so that $$\gamma^{1/2}_\Theta(E_1)\ge\frac{\mu(E_1)}{\pi}=\frac{L}{\pi}>0.$$
\vv

\item[b)] Assume that $\gamma_\Theta^{1/2}(E_2)>0$. Then there exists a distribution $S$ supported on $E_2$ such that $\|P\ast S\|_{\infty}\le 1.$ Approximating the distribution $S$ by signed measures (take for instance the convolution of $S$ with approximations of the identity in the variable $t$), we can assume that there exists a signed measure $\nu$ supported on the segment $$E_{\varepsilon}=\{(0,t):-\varepsilon\le t\le L+\varepsilon\}$$ for some small $\varepsilon>0$, satisfying $\|P\ast \nu\|_{\infty}\le 2.$ Since $P$ is a non negative kernel, there exists a positive measure $\mu$ supported on $E_{\varepsilon}$ with $\mu(E_{\varepsilon})>0$ such that 
\begin{equation}
\label{a}
\|P\ast \mu\|_{\infty}\le 2.
\end{equation}
To get a contradiction we will show that \eqref{a} implies $\mu(E_{\varepsilon})=0.$ 
Since, by lemma \ref{lemgrow}, the measure $\mu$ has linear growth, given $\eta>0$ we can take $c\in(-\ve,L+\ve)$ such that $$\mu(\{(0,t): c\le t \le L
+\ve\})<\eta.$$ Set $F=\{(0,t): -\ve\le t \le c\}$. By \eqref{a}, for any $(0,s)\in F$ there exists $\ell=\ell(0,s)>0$ such that 
$$
\int_s^{s+\ell}\frac{t-s}{(t-s)^2}d\mu(t)\le \eta.
$$
Hence, since $F$ is an interval,  there exists a finite number of almost disjoint intervals $I_{n}=\{0\}\times[s_n, s_n+\ell_n], \, 1\le n\le N,$ such that
$\displaystyle{F\subset\bigcup_{n=1}^{N} I_n}$ and 
$$
\mu(I_n)=\int_{s_n}^{s_n+\ell_n}d\mu(t)\le\int_{s_n}^{s_n+\ell_n}\frac{\ell_n}{t-s_n}d\mu(t)\le \eta \ell_n. 
$$
Consequently, 
$$
\mu(E_{\varepsilon})\le \mu(F)+\eta\lesssim \sum_{n=1}^{N} \mu(I_n)+\eta\le  \eta (\sum_{n=1}^{N}\ell_n+1)\le \eta(c+\ve+1),
$$
which means $\mu(E_{\varepsilon})=0$ and we get the desired contradiction.  
\end{enumerate}

 \end{proof}

\vv
\section{$s$-Capacities and $s$-growth for $s\in(1/2,1)$.}\label{section-s1}
Let $\|\cdot\|_{*,p}$ denote the norm of the parabolic BMO space:
$$\|f\|_{*,p} = \sup_Q \avint_Q |f -m_Q f|\,dm,$$
where the supremum is taken over all $s-$parabolic cubes $Q\subset \R^{N+1}$,
$dm$ stands for the Lebesgue measure in $\R^{N+1}$ and $m_Qf$ is the mean of $f$ with respect to $dm$. \newline

For a function $f:\R^{N+1}\to\R$ and $\alpha\in(0,1)$, the $\alpha-$fractional derivative with respect to $t$ (recall that $x\in\R^N$ and $t\in\R$) is defined as
$$
\partial_t^{\alpha}f(x,t)=\int\frac{f(x,s)-f(x,t)}{|s-t|^{1+\alpha}}ds.
$$

Now we introduce the $\gamma_\Theta^s$ capacities for $s\in (1/2,1)$. Given a compact set $E\subset\R^{N+1}$, we define 
\begin{equation}\label{caps21}
\gamma_{\Theta}^s(E)=\sup\{|\langle\nu,1\rangle|\},
\end{equation}
where the supremum is taken over all distributions $\nu$ supported on $E$ satisfying 
 \begin{equation}\label{potentials1}
 \|(-\Delta)^{s-\frac12}(\nu*P_s)\|_\infty\leq 1\quad\mbox{ and}\quad\|\partial_t^{1-\frac1{2s}}(\nu*P_s)\|_{*,p}\leq 1.
 \end{equation}
 \vv
 Notice that for $s=1/2$ we obtain $\gamma_\Theta^{1/2}$.

  The capacity $\gamma_\Theta^s$ is called the $s$-fractional caloric capacity of $E$. We set $\gamma_{\Theta,+}^s(E)$ as in \eqref{caps21} but with the supremum restricted to all positive measures $\nu$ supported on $E$ satisfying \eqref{potentials1}. Clearly $$\gamma_{\Theta,+}^s(E)\leq\gamma_\Theta^s(E).$$
  We are now ready to describe the basic relationship between the capacity $\gamma_\Theta^s$ and Hausdorff content (the $d$-dimensional Hausdorff content will be denoted by ${\mathcal H}_\infty^d$).
\begin{theorem}\label{hausdorffcontent}
 Let $s\in (1/2,1)$. For every Borel set $E\subset\R^{N+1}$, $$\gamma_{\Theta,+}^s(E)\leq\gamma^s_{\Theta}(E)\lesssim {\mathcal H}^{N+2s-1}_{\infty}(E)$$ and $${\mbox{dim}}_{H,p}(E)>N+2s-1\Longrightarrow \gamma_{\Theta}^s(E)>0.$$\end{theorem}

To prove theorem \ref{hausdorffcontent} we need to study the growth conditions satisfied by positive measures and distributions with properties \eqref{potentials1}.

\subsection{The case of positive measures.}
We state first two lemmas concerning the case when the distribution $\nu$ is a positive measure. But before going to the next lemma, recall that a function $f(x,t)$ defined in $\R^{N+1}$ is Lip$(\alpha)$, $0\alpha<1$, in the $t$ variable if $$\|f\|_{Lip(\alpha)}=\sup_{x\in\R^N,t,u\in\R}\frac{|f(x,t)-f(x,u)|}{|t-u|^{\alpha}}<\infty.$$
\begin{lemma} \label{lemlip}
Let $s\in (1/2,1)$ and let $\mu$ be a positive measure in $\R^{N+1}$ with upper parabolic growth of degree $N+2s-1$ with constant $1$. Then,
$$\|P_s*\mu\|_{Lip(1-\frac1{2s})}\lesssim 1.$$
\end{lemma} 

\begin{proof}
Let $\bar x = (x,t)$, $\hat x = (x,u)$, and $\bar x_0 = \frac12(\bar x + \hat x)$.
Then, writing $\bar y=(y,s)$, we split
\begin{align*}
|P_s*\mu(\bar x) - P_s*\mu(\hat x)| & \leq \int_{|\bar y -\bar x_0|_p\geq 2 |\bar x-\hat x|_p}
|P_s(x-y,t-s) - P_s(x-y,u-s)|\,d\mu(\bar y) \\
&+ \int_{|\bar y -\bar x_0|_p < 2 |\bar x-\hat x|_p}\!
|P_s(x-y,t-s) - P_s(x-y,u-s)|\,d\mu(\bar y)\\& =: I_1 + I_2.
\end{align*}

To shorten notation, we write $d:=|\bar x-\hat x|_p = |t-u|^{\frac1{2s}}$. Then we have
$$I_1\lesssim \sum_{k\geq 1} \int_{2^k d\leq |\bar y -\bar x_0|_p< 2^{k+1}d} \,\,\sup_{\xi\in [\bar x-\bar y,
\hat x-\bar y]} |\partial_t P_s(\xi)| \,|t-u|\,d\mu(\bar y).$$
Since by property $(3)$ in lemma \ref{bounds}
$$|\partial_t P_s (\xi)|\lesssim \frac1{|\xi|_p^{N+2s}}\approx \frac1{(2^kd)^{N+2s}}\quad
\mbox{ if $\xi\in [\bar x-\bar y,
\hat x-\bar y]$, $|\bar y -\bar x_0|_p\approx 2^{k}d$,}$$
we deduce that
$$I_1\lesssim \sum_{k\geq 1} \frac{\mu(B_p(\bar x_0,2^{k+1}d))}{(2^k d)^{N+2s}}\,|t-u|\lesssim \frac{|t-u|}d= |t-u|^{1-\frac1{2s}}.$$

Next we deal with $I_2$. Writing $B_0= B_p(\bar x_0,2d)$, we have
$$I_2 \leq P_s*(\chi_{B_0}\mu)(\bar x) + P_s*(\chi_{B_0}\mu)(\hat x).$$
Observe now that by lemma \ref{bounds}
\begin{align*}
0&\leq P_s*(\chi_{B_0}\mu)(\bar x) \lesssim \int_{\bar y\in B_0} \frac1{|\bar x- \bar y|_p^N}\,
d\mu(\bar y) 
\\&\leq \int_{|\bar x-\bar y|_p\leq 4d} \frac1{|\bar x- \bar y|_p^N}\,
d\mu(\bar y)
\lesssim d^{2s-1} =|t-u|^{1-\frac1{2s}}.
\end{align*}
The last estimate follows by splitting the integral into parabolic annuli and using the parabolic
growth of order $N+2s-1$ of $\mu$, for example.
The same estimate holds replacing $\bar x$ by $\hat x$. Then gathering all the estimates above, the lemma
follows.
\end{proof}
\begin{lemma}\label{growthbmo}
Let $s\in(1/2,1)$ and let $\mu$ be a positive measure in $\R^{N+1}$ which has upper parabolic growth $N+2s-1$, with constant $1$. Then $$\|\partial_t^{1-\frac1{2s}}P_s*\mu\|_{*,p}\lesssim 1.$$
\end{lemma}
\begin{proof}
Let $Q\subset \R^{n+1}$ be a fixed 
$s$-parabolic cube. We have to show that there exists some constant $c_Q$
(to be chosen below) such that
$$\int_Q |\partial_t^{1-\frac1{2s}} P_s*\mu - c_Q|\,dm \lesssim \ell(Q)^{N+2s}.$$
To this end consider a $C^\infty$ test function $\wt\chi_{5Q}$ which equals $1$ on $5Q$ and vanishes in $6Q^c$.

We also denote $\wt\chi_{5Q^c} = 1 - \wt\chi_{5Q}$. Then we
split
\begin{align*}
\int_Q |\partial_t^{1-\frac1{2s}} P_s*\mu - c_Q|\,dm &\leq
\int_Q |\partial_t^{1-\frac1{2s}} P_s*(\wt\chi_{5Q}\mu)|\,dm \\&\quad+ \int_Q |\partial_t^{1-\frac1{2s}} P_s*(\wt\chi_{5Q^c}\mu) - c_Q|\,dm =: I_1 + I_2.
\end{align*}
To deal with the integral $I_1$, we use lemma \ref{bounds} and write
$$
\int_Q |\partial_t^{1-\frac1{2s}} P_s*(\wt\chi_{5Q}\mu)|\,dm\lesssim
\int_Q \frac1{|x|^{N-1}\,|\bar x|_p^{2s}}* (\wt\chi_{5Q}\mu)\,dm \leq \int_{6Q} \frac1{|y|^{N-1}\,|\bar y|_p^{2s}}* (\chi_{Q}m)\,d\mu.$$
Taking into account that, for $\bar y=(y,u)$,
$$\frac1{|y|^{N-1}\,|\bar y|_p^{2s}}\leq \frac1{|y|^{N-1+s/2}}\,\frac1{u^{3/4}},$$
and writing $Q=Q_1\times I_Q$, where $Q_1$ is a cube with side length $\ell(Q)$ in $\R^N$ and $I_Q$ is an interval of length $\ell(Q)^{2s}$,
we deduce that for $\bar x=(x,t)\in 6Q$,
\begin{align*}
\frac1{|y|^{N-1}\,|\bar y|_p^{2s}}* (\chi_{Q}m)(\bar x) &\lesssim \int_{y\in Q_1} \frac1{|x-y|^{N-1+s/2}}dy\,\int_{u\in
I_Q}\frac1{|t-u|^{3/4}}\,du \\&\lesssim \ell(Q)^{1-s/2}\,(\ell(Q)^{2s})^{1/4} = \ell(Q).
\end{align*}
Thus,
$$\int_Q |\partial_t^{1-\frac1{2s}} P_s*(\wt\chi_{5Q}\mu)|\,dm\lesssim \ell(Q)\,\mu(6Q)\lesssim \ell(Q)^{N+2s}.$$

Next we will estimate the integral $I_2$, taking $c_Q:=
\partial_t^{1-\frac1{2s}} P_s*(\wt\chi_{5Q^c}\mu)(\bar x_Q)$, where $\bar x_Q$ is the center of $Q$.
To show that $I_2\lesssim \ell(Q)^{N+2s}$, it suffices to prove
$$|\partial_t^{1-\frac1{2s}} P_s*(\wt\chi_{5Q^c}\mu)(\bar x) - 
\partial_t^{1-\frac1{2s}} P_s*(\wt\chi_{5Q^c}\mu)(\bar x_Q)|\lesssim 1.$$
In turn, to prove this it is enough to show that
\begin{equation}
|\partial_t^{1-\frac1{2s}} P_s*(\wt\chi_{5Q^c}\mu)(\bar x) - 
\partial_t^{1-\frac1{2s}} P_s*(\wt\chi_{5Q^c}\mu)(\bar y)|\lesssim 1\label{eqdif}
\end{equation}
for $\bar x,\bar y\in \R^{N+1}$ in the following two cases:
\begin{itemize}
\item Case 1:  $\bar x,\bar y\in Q$ of the form $\bar x= (x,t)$, $\bar y= (y,t)$.
\item Case 2:  $\bar x,\bar y\in Q$ of the form $\bar x= (x,t)$, $\bar y= (x,u)$.
\end{itemize}

\vv
\noi {\em Proof of \rf{eqdif} in Case 1.}
Let $\phi=\wt\chi_{5Q^c}$. We split
\begin{align*}
&|\partial_t^{1-\frac1{2s}} P_s* (\phi\mu)(x,t) - \partial_t^{1-\frac1{2s}} P_s* (\phi\mu)(y,t)|\\
 & =
\left|\int \frac{P_s*(\phi\mu)(x,u) - P_s*(\phi\mu)(x,t)}{|u-t|^{2-\frac1{2s}}}du -\!
\int \frac{P_s*(\phi\mu)(y,u) - P_s*(\phi\mu)(y,t)}{|u-t|^{2-\frac1{2s}}}du\right|\\
& \leq\quad\quad
\int_{|u-t|\leq2^{2s}\ell(Q)^{2s}} \frac{|P_s*(\phi\mu)(x,u) - P_s*(\phi\mu)(x,t)|}{|u-t|^{2-\frac1{2s}}}\,du\\
&\quad\quad  +
\int_{|u-t|\leq2^{2s}\ell(Q)^{2s}} \frac{|P_s*(\phi\mu)(y,u) - P_s*(\phi\mu)(y,t)|}{|u-t|^{2-\frac1{2s}}}\,du
\\
&\quad\quad+
\!\int_{|u-t|>2^{2s}\ell(Q)^{2s}}\!\!\!\! \frac{|P_s\!*(\phi\mu)(x,u) - P_s\!*(\phi\mu)(x,t) - P_s\!*(\phi\mu)(y,u) + P_s\!*(\phi\mu)(y,t)|}{|u-t|^{2-\frac1{2s}}}du\\
& =: A_1  + A_2 + B.
\end{align*}

First we will estimate the term $A_1$ . For $u,t$ such that $|u-t|\leq2^{2s}\ell(Q)^{2s}$, we write
$$|P_s*(\phi\mu)(x,u) - P_s*(\phi\mu)(x,t)|\leq |u-t|\,\|\partial_t P_s*(\phi\mu)\|_{\infty,(1+2^{2s})Q}.$$
Observe now that for each $\bar z= (z,v)\in (1+2^{2s})Q$ (see property $(3)$ in lemma \ref{bounds})
$$|\partial_t P_s*(\phi\mu)(\bar z)| \lesssim\int_{5Q^c} \frac1{|\bar z-\bar w|_p^{N+2s}}\,d\mu(\bar w)
\lesssim \frac1{\ell(Q)},$$
by splitting the last domain of integration into parabolic annuli and using the growth condition of order $N+2s-1$ of $\mu$. Thus,
 $$|P_s*(\phi\mu)(x,u) - P_s*(\phi\mu)(x,t)|\lesssim \frac{|u-t|}{\ell(Q)}.$$
 Plugging this into the integral that
defines $A_1$, we obtain
$$A_1\lesssim \int_{|u-t|\leq2^{2s}\ell(Q)^{2s}} \frac{|u-t|}{\ell(Q)\,|u-t|^{2-\frac1{2s}}}\,ds 
\lesssim \frac{(\ell(Q)^{2s})^{\frac1{2s}}}{\ell(Q)}=1.$$
By exactly the same arguments, just writing $y$ in place of $x$ above, we deduce also that
$$A_2\lesssim 1.$$

Concerning the term $B$, we write
\begin{align*}
B & \leq 
\int_{|u-t|>2^{2s}\ell(Q)^{2s}} \frac{|P_s*(\phi\mu)(x,t) - P_s*(\phi\mu)(y,t)|}{|u-t|^{2-\frac1{2s}}}\,du\\
&\quad
+
\int_{|u-t|>2^{2s}\ell(Q)^{2s}} \frac{|P_s*(\phi\mu)(x,u) - P_s*(\phi\mu)(y,u)|}{|u-t|^{2-\frac1{2s}}}\,du=B_1+B_2.
\end{align*}
By Lemma \ref{bounds}, it follows that for $\bar \xi\in Q$, $$|\nabla_x P_s*(\phi\mu)(\bar \xi)|\lesssim\int_{5Q^c}\frac{d\mu(\bar z)}{|\bar\xi-\bar z|^{N+1}_p}\lesssim \ell(Q)^{2s-2}.$$
Therefore,
$$|P_s*(\phi\mu)(x,t) - P_s*(\phi\mu)(y,t)|\leq \|\nabla_x P_s*(\phi\mu)\|_{\infty,Q} 
\,|x-y|\lesssim \ell(Q)^{2s-1}.$$
Hence $B_1\lesssim 1.$
For $B_2$, we split the integral into annulus $\mathcal{A}_k=\{u\in\R:|u-t|\approx (2^k\ell(Q))^{2s}\}$ and write
\begin{align*}
B_2&=\int_{|u-t|>2^{2s}\ell(Q)^{2s}} \frac{|P_s*(\phi\mu)(x,u) - P_s*(\phi\mu)(y,u)|}{|u-t|^{2-\frac1{2s}}}\,du\\&=\sum_k\int_{\mathcal{A}_k}\frac{|P_s*(\phi\mu)(x,u) - P_s*(\phi\mu)(y,u)|}{|u-t|^{2-\frac1{2s}}}\,du.
\end{align*}
For $u\in\mathcal{A}_k,$ using lemma \ref{bounds} we get
\begin{align*}
&|P_s*(\phi\mu)(x,u)- P_s*(\phi\mu)(y,u)|\\&\lesssim \underset{5Q^c\cap\{|\bar z-(x_Q,u)|\leq 3\ell(Q)\} }{\int}\frac{d\mu(\bar z)}{|(x_Q,u)-\bar z|_p^N}+\underset{5Q^c\cap\{|\bar z-(x_Q,u)|> 3\ell(Q)\} }{\int}\big|P_s((x,u)-\bar z)-P_s((y,u)-\bar z)\big|d\mu(\bar z)\\&\lesssim \underset{5Q^c\cap\{|\bar z-(x_Q,u)|\leq 3\ell(Q)\} }{\int}\frac{d\mu(\bar z)}{|(x_Q,u)-\bar z|_p^N}+\underset{5Q^c\cap\{|\bar z-(x_Q,u)|> 3\ell(Q)\} }{\int}\frac{|x-y|}{|(x_Q,u)-\bar z|_p^{N+1}}d\mu(\bar z)\\&\lesssim \ell(Q)^{2s-1}+\ell(Q)\ell(Q)^{2s-2}=\ell(Q)^{2s-1}.
\end{align*}

Therefore

$$B_2\lesssim \int_{|u-t|>2^{2s}\ell(Q)^{2s}} \frac{\ell(Q)^{2s-1}}{|u-t|^{2-\frac1{2s}}}\,du \lesssim \frac{\ell(Q)^{2s-1}}{(\ell(Q)^{2s})^{1-\frac1{2s}}}\lesssim 1.$$
So \rf{eqdif} holds in this case.

\vv
\noi {\em Proof of \rf{eqdif} in Case 2.}
As in Case 1, we write
\begin{align*}
&|\partial_t^{1-\frac1{2s}} P_s* (\phi\mu)(x,t) - \partial_t^{1-\frac1{2s}} P_s* (\phi\mu)(x,u)|\\
 & =
\biggl|\int \frac{P_s*(\phi\mu)(x,v) - P_s*(\phi\mu)(x,t)}{|v-t|^{2-\frac1{2s}}}dv -\!
\int \frac{P_s*(\phi\mu)(x,v) - P_s*(\phi\mu)(x,u)}{|v-u|^{2-\frac1{2s}}}dv\biggr|\\
& \leq\quad\quad
\int_{|v-t|\leq2^{2s}\ell(Q)^{2s}} \frac{|P_s*(\phi\mu)(x,v) - P_s*(\phi\mu)(x,t)|}{|v-t|^{2-\frac1{2s}}}\,dv\\
&\quad\quad+
\int_{|v-t|\leq2^{2s}\ell(Q)^{2s}} \frac{|P_s*(\phi\mu)(x,v) - P_s*(\phi\mu)(x,u)|}{|v-u|^{2-\frac1{2s}}}\,dv
\\
&\quad\quad+
\int_{|v-t|>2^{2s}\ell(Q)^{2s}} \biggl|\frac{P_s*(\phi\mu)(x,v) - P_s*(\phi\mu)(x,t)}{|v-t|^{2-\frac1{2s}}}-\frac{P_s*(\phi\mu)(x,v) - P_s*(\phi\mu)(x,u)}{|v-u|^{2-\frac1{2s}}}\biggr|dv\\
& =: A_1'  + A_2' + B'.
\end{align*}
The terms $A_1'$ and $A_2'$ can be estimated exactly in the same way as the terms $A_1$ and $A_2$ in Case 1, so that
$$A_1' + A_2'\lesssim 1.$$

Concerning $B'$, we have
\begin{align*}
B'&\leq 
 \int_{|v-t|>2^{2s}\ell(Q)^{2s}} 
\left|\frac1{|v-t|^{2-\frac1{2s}}} - \frac1{|v-u|^{2-\frac1{2s}}}\right|\,
 \big|P_s*(\phi\mu)(x,v) - P_s*(\phi\mu)(x,t)\big| \,dv\\
&\qquad 
+ \int_{|v-t|>2^{2s}\ell(Q)^{2s}} 
\frac1{|v-u|^{2-\frac1{2s}}}\,
 \big|P_s*(\phi\mu)(x,t) - P_s*(\phi\mu)(x,u)\big| \,dv
\end{align*}
Taking into account that, for $|v-t|>2^{2s}\ell(Q)^{2s}$,
$$\left|\frac1{|v-t|^{2-\frac1{2s}}} - \frac1{|v-u|^{2-\frac1{2s}}}\right|\lesssim \frac{|t-u|}{|v-t|^{3-\frac1{2s}}}
\lesssim \frac{\ell(Q)^{2s}}{|v-t|^{3-\frac1{2s}}}$$
and that, by Lemma \ref{lemlip}, $P_s*(\phi\mu)(x,\cdot)$ is Lip$(1-\frac1{2s})$ in the variable $t$, 
we deduce that
$$B' \lesssim 
  \ell(Q)^{2s}\int_{|v-t|>2^{2s}\ell(Q)^{2s}} \frac{|v - t|^{1-\frac1{2s}}}{|v-t|^{3-\frac1{2s}}}
 \,dv\\
+ 
 \int_{|v-t|>2^{2s}\ell(Q)^{2s}}\frac1{|v-u|^{2-\frac1{2s}}}\,|t-u|^{1-\frac1{2s}} \,dv \lesssim 1.
$$

\vv
This concludes the proof of $\displaystyle{\|\partial_t^{1-\frac1{2s}}P_s*\mu\|_{*,p}\lesssim 1}$ for $s\in (1/2,1)$. 
\end{proof}
\subsection{Growth condition and admissible functions.}
When $s\in(1/2,1),$ we say that a ${\mathcal C}^{1}$ function $\vphi$ is admissible for an $s$-parabolic cube $Q=Q_1\times I_Q$, $Q_1\subset\R^N$ and $I_Q\subset\R$, if it is supported on $Q\subset\R^{N+1}$ and satisfies 
\begin{equation}\label{admissiblephix}
\int_{\R^{N+1}}|(-\Delta_x)^{1/2}\vphi(x,t)|dxdt\lesssim\ell(Q)^{N+2s-1}
\end{equation} 
and 
\begin{equation}\label{admissiblephit}
\|\partial_t\vphi\|_\infty\leq\frac1{\ell(I_Q)}=\frac1{\ell(Q)^{2s}}.
\end{equation}
As in the case $s=1/2$ (see \rf{admissiblephi}) recall that $\displaystyle(-\Delta_x)^{1/2}\vphi\approx\sum_{j=1}^NR_j\partial_j\vphi,$ with $R_j$ being the Riesz transforms with Fourier multiplier $\xi_j/|\xi|$. Then, for an $s$-parabolic cube $Q\subset\R^{N+1},$ $Q=Q_1\times I_Q$ with $Q_1\subset\R^N$ and $I_Q\subset\R$, one can write
\begin{align*}
\int_{\R^{N+1}}|(-\Delta_x)^{1/2}\vphi(x,t)|dxdt&\approx\int_{I_Q}\int_{\R^N}|\sum_{j=1}^NR_j\partial_j\vphi(x,t)|dxdt\\&\leq\sum_{j=1}^N\int_{I_Q}\int_{\R^N}|R_j\partial_j\vphi(x,t)|dxdt
\end{align*}
Therefore if $\vphi$ satisfies  \rf{admissiblephit} and $$\int_{I_Q}\int_{\R^N}|R_j\partial_j\vphi(x,t)|dxdt\lesssim\ell(Q)^{N+2s-1},\;\mbox{ for }1\leq j\leq N,$$ then $\vphi$ is admissible for the $s$-parabolic cube $Q$, $s\in (1/2,1)$.

Let's remark that  the analogues of lemmas \ref{admissibility} and \ref{standardtest} also hold in this context and if $\varphi$ is a standard test function supported on an $s-$parabolic cube $Q=Q_1\times I_Q\subset\R^{N+1}$ with $\|\partial_t\varphi\|_\infty\lesssim\ell(I_Q)^{-1}=\ell(Q)^{-2s}$ and $\|\Delta_x\varphi\|_{\infty}\lesssim\ell(Q)^{-2}$, then $\varphi$ is admissible for $Q$. In fact, arguing as in \eqref{standardtestok} we have that 
\begin{align*}
\int_{\R^{N+1}}|(-\Delta_x)^{1/2}\vphi(x,t)|dxdt&\approx\int_{\R^{N+1}}|\Delta_x\varphi*_x k(x,t)|dxdt\\&\lesssim\int_{I_Q}\int_{\R^N\setminus2Q_1}\int_{Q_1}\frac{|\varphi(y,t)|}{|x-y|^{N+1}}dydxdt+\int_{I_Q}\int_{2Q_1}\int_{Q_1}\frac{|\Delta_y\varphi(y,t)|}{|x-y|^{N-1}}dydxdt\\&\lesssim\sum_{k=1}^\infty\frac{\ell(Q)^{N+2s}(2^k\ell(Q))^{N}}{(2^k\ell(Q))^{N+1}}+\ell(Q)^{N+2s-1}\lesssim\ell(Q)^{N+2s-1}.
\end{align*}
For these values of $s$, the analogous growth condition to lemma \ref{lemgrow} reads as follows.
\begin{lemma}\label{case1/21}Let $s\in [\frac12,1)$ and $\nu$ a distribution such that $$ \|(-\Delta)^{s-\frac12}(\nu*P_s)\|_\infty\leq 1\quad\mbox{ and}\quad\|\partial_t^{1-\frac1{2s}}(P_s*\nu)\|_{*,p}\leq 1.$$ If $\vphi$ is an admissible function for an $s-$parabolic cube $Q\subset\R^{N+1}$, then $$\displaystyle \left|\langle \nu,\varphi\rangle \right|\lesssim\ell(Q)^{N+2s-1}.$$
\end{lemma}
\begin{proof}
\begin{align*}
|\langle\nu,\vphi\rangle|&=|\langle\nu,\Theta^s\vphi*P_s\rangle|\leq|\langle\nu*P_s,(-\Delta)^s\vphi\rangle|+|\langle\nu*P_s,\partial_t\vphi\rangle|=I_1+I_2.
\end{align*}
Notice that since $\frac12\leq s<1$, if we consider the Fourier transform of $\displaystyle(-\Delta)^s\vphi$, we get $$\widehat{\left((-\Delta)^s\vphi\right)}(\xi)=|\xi|^{2s}\hat\vphi=|\xi|^{2s-1}|\xi|\hat\vphi.$$ Hence we can write $(-\Delta)^s\vphi=(-\Delta)^{s-\frac12}(-\Delta)^{\frac12}\vphi$. Therefore, since $\vphi$ is admissible for the $s$-parabolic cube $Q$, \rf{admissiblephix} is satisfied and so
\begin{align*}
I_1&=|\langle \nu*P_s,(-\Delta)^s\vphi\rangle|=|\langle (-\Delta)^{s-\frac12}(\nu*P_s),(-\Delta)^\frac12\vphi\rangle|\\&\leq \|(-\Delta)^{s-\frac12}(\nu*P_s)\|_\infty\int_{I_Q\times\R^N}\big|(-\Delta)^\frac12\vphi(x,t)\big|dxdt\\&\lesssim\ell(Q)^{N+2s-1}.
\end{align*}
We deal now with $I_2$. Notice that $$\partial_t\vphi=c\;\partial_t^{1-\frac1{2s}}(\partial_t\vphi*\frac1{|t|^{1/{2s}}}),$$ for some constant $c$. This identity can be easily seen by taking the Fourier transform (on the variable $t$). Set $\displaystyle g=\partial_t\vphi*\frac1{|t|^{1/{2s}}}$. Then
$$I_2=|\langle\nu*P_s,\partial_t\vphi\rangle|=|c||\langle \partial_t^{1-\frac1{2s}}(\nu*P_s),g\rangle|.$$
Write $Q=Q_1\times I_Q$, with $Q_1\subset\R^N$ and $I_Q\subset\R$, and let $c_Q$ be the center of the interval $I_Q$. Because of the zero mean of $\partial_t\vphi$ (integrating with respect to $t$) it is easy to check that $|g(x,t)|$ decays at most like $\displaystyle |t|^{-1-\frac1{2s}}$ at infinity. Indeed,  for $t\notin 2I_Q$ we have
\begin{align}\label{out}
|g(x,t)|&=\left|\int_{I_Q}\frac{\partial_u\vphi(x,u)}{|t-u|^{\frac1{2s}}}du\right|=\left|\int_{I_Q}\partial_u\vphi(x,u)\left(\frac1{|t-u|^{\frac1{2s}}}-\frac1{|t-c_Q|^{\frac1{2s}}}\right)du\right|\\&\lesssim\frac{\ell(I_Q)}{|t-c_Q|^{1+\frac1{2s}}}\int_{I_Q}|\partial_u\vphi(x,u)|du\lesssim\frac{\ell(I_Q)}{|t-c_Q|^{1+\frac1{2s}}}=\frac{\ell(Q)^{2s}}{|t-c_Q|^{1+\frac1{2s}}},\nonumber
\end{align}
where $c_Q$ is the center of $Q$.
And for $t\in4I_Q$, \begin{equation}\label{in}|g(x,t)|\lesssim\|\partial_u\vphi\|_\infty\ell(I_Q)^{1-\frac1{2s}}\lesssim\ell(I_Q)^{-\frac1{2s}}=\ell(Q)^{-1}.\end{equation}
Since $\int g\;dm=0$, writing $f=\partial_t^{1-\frac1{2s}}(\nu*P_s)$, we have
\begin{align*}
I_2&=|c||\langle f,g\rangle|=\left|c\int(f-m_Qf)\;g\;dm\right|\\&\lesssim\int_{2Q}|f-m_Qf|\;|g|\;dm+\int_{\R^{N+1}\setminus 2Q}(f-m_Qf)\;g\;dm\\&=I_{21}+I_{22}.
\end{align*}
Hence, due to \eqref{in}, $$I_{21}\lesssim\|f\|_{*,p}\,\ell(Q)^{N+2s}\ell(Q)^{-1}\leq\ell(Q)^{N+2s-1}.$$
For $I_{22}$, we split the domain of integration in annuli. Write $A_i= 2^{i} Q\setminus 2^{i-1} Q$ for $i\geq1$.
Remark that for an $s$-parabolic cube $Q = Q_1\times I_Q$, we denote
$$2^i Q = 2^i Q_1 \times 2^{2si} I_Q,$$
so that $2^iQ$ is an $s$-parabolic cube too (notice that if $Q$ is centered at the origin and we consider the parabolic dilation $\delta_\lambda(x,t) = (\lambda x,\lambda^{2s}t)$, $\lambda>0$, we have $2^iQ = \delta_{2^i}(Q)$).
Then, using the decay of $g$ given by \rf{out}, we get
\begin{equation}\label{eq3434}
I_{22}\lesssim\sum_{i=1}^\infty \frac{\ell(Q)^{2s}}{\ell(2^iQ)^{1+2s}}\left(\int_{A_i\cap \supp g}|f-m_{2^{i}Q}f|\,dm+ \int_{A_i\cap \supp g}|m_{2^{i}Q}f-m_{Q}f|\,dm\right).
\end{equation}
To estimate the first integral on the right hand side, recall that 
$\supp g\subset Q_1\times \R$. Using H\"older's inequality with some 
exponent $q\in(0,\infty)$ to be chosen in a moment and the fact that
$f\in BMO_p$ (together with John-Nirenberg), then we get:
\begin{align*}
\int_{A_i\cap \supp g}|f-m_{2^{i}Q}f|\,dm 
&\leq \left(\int_{2^iQ}|f-m_{2^{i}Q}f|^q\,dm\right)^{\frac1q} \,m(\supp g\cap 2^iQ)^{\frac1{q'}}\\
& \lesssim \ell(2^iQ)^{\frac{N+2s}q}\,(\ell(Q)^N\,\ell(2^iQ)^{2s})^{\frac 1{q'}} \\&= \ell(2^iQ)^{\frac Nq+ 2s}\,\ell(Q)^{\frac N{q'}}.
\end{align*}
For the last integral on the right hand side of \rf{eq3434}, we write
$$\int_{A_i\cap\supp g}|m_{2^{i}Q}f-m_{Q}f|\,dm\lesssim i\,m(2^iQ\cap \supp g) \leq i\,\ell(Q)^N\,\ell(2^iQ)^{2s}.$$
Therefore,
\begin{align*}
I_{22}&\lesssim\sum_{i=1}^\infty \frac{\ell(Q)^{2s}}{\ell(2^iQ)^{1+2s}}\,\Big(
\ell(2^iQ)^{\frac Nq+ 2s}\,\ell(Q)^{\frac N{q'}} + i\,\ell(Q)^N\,\ell(2^iQ)^{2s}\Big)\\&=\frac{\ell(Q)^{2s}\ell(Q)^{N+2s}}{\ell(Q)^{1+2s}}\sum_{i=1}^\infty \,
2^{i(\frac Nq-1)}(1+\frac{i}{2^i}).
\end{align*}
Choosing $q>N$, we get
$$I_{22}\lesssim \ell(Q)^{N+2s-1}.$$

\end{proof}
\subsection{Proof of Theorem \ref{hausdorffcontent}}
The inequality $\gamma_{\Theta,+}^s(E)\leq\gamma_{\Theta}^s(E)$ comes from the definition of the capacities. For the second inequality we assume $E$ to be compact and let $\{Q_j\}_j$ be a covering of $E$ by $s$-parabolic cubes $Q_j\subset\R^{N+1}$ with disjoint interiors. By a parabolic version of a well known lemma of Harvey and Polking (see \cite{Harvey-Polking}), there exist functions $g_j\in\mathcal{C}_0^\infty(2Q_j)$ satisfying $\sum_jg_j=1$ in a neighborhood of $\cup_jQ_j$, \eqref{admissiblephix} and \eqref{admissiblephit}. 

Let $\nu$ be a distribution with compact support contained in $E$ such that \eqref{potentials1} hold. Then, using lemma \ref{case1/21}
$$
\big|\langle \nu,1\rangle\big|\leq\sum_j|\langle\nu,g_j\rangle|\leq\sum_j\ell(Q_j)^{N+2s-1}.
$$
Thus, $\gamma_{\Theta}^s(E)\leq C{\mathcal H}^{N+2s-1}_{\infty}(E)$.

The second assertion in the theorem follows a standard argument that we reproduce for the reader's convenience. Suppose $E\subset\R^{N+1}$ is a Borel set with dim$_{H,p}(E)=\alpha>N+2s-1$. By a parabolic version of Frostman's Lemma (that can be proved by arguments analogous to classical ones replacing the usual dyadic lattice in $\R^{N+1}$ by a parabolic dyadic lattice), there exists a non-zero positive measure $\mu$ supported on $E$ such that $\mu(B(\bar x,r))\leq r^{\alpha}$ for all $\bar x\in\R^{N+1}$ and $r>0$. We have to show estimates \eqref{potentials1} with $\nu$ replaced by $\mu$.
\begin{itemize}
\item The parabolic BMO-estimate in \eqref{potentials1} is just a consequence of lemma \ref{growthbmo}. Therefore $\displaystyle{\|\partial_t^{1-\frac1{2s}}P_s*\mu\|_{*,p}\lesssim 1}$.
\item To prove $\displaystyle \|\Delta^{s-\frac12}P_s*\mu\|_{\infty}\lesssim 1$, apply the second property in lemma \ref{bounds},
\begin{align}\label{slaplacian}
|\Delta^{s-\frac12}P_s*\mu(\bar y)|&\lesssim \int\frac{d\mu(\bar x)}{|\bar x-\bar y|^{N+2s-1}}=\int_0^\infty\mu(\{\bar x: |\bar x-\bar y|^{-N-2s+1}\geq u\})du\nonumber\\&=\int_0^\infty\mu(\{\bar x: |\bar x-\bar y|\leq u^{\frac{-1}{N+2s-1}}\})du=\int_0^{\infty}\mu(B(\bar y,u^{\frac{-1}{N+2s-1}}))du\\&\approx\int_0^{\infty}\frac{\mu(B(\bar y,r))}{r^{N+2s}}dr\lesssim\int_0^{\mu(E)^{1/\alpha}}\frac{r^{\alpha}}{r^{N+2s}}dr+\int_{\mu(E)^{1/\alpha}}^\infty\frac{\mu(E)}{r^{N+2s}}dr\nonumber\\&\approx \mu(E)^{\frac{\alpha-(N+2s-1)}{\alpha}}\nonumber.
\end{align}

\end{itemize}
\qed

\vv

\section{$s-$capacities and $s$-growth in case $s\in(0,1/2)$.}\label{section-s2}

Now we introduce the $\gamma_\Theta^s$ capacities for $s\in (0,1/2)$. Given a compact set $E\subset\R^{N+1}$, we define 
\begin{equation}\label{caps2}
\gamma_{\Theta}^s(E)=\sup\{|\langle\nu,1\rangle|\},
\end{equation}
where the supremum is taken over all distributions $\nu$ supported on $E$ satisfying properties \eqref{potentials2} listed below. Before we write them, to get some intuition, let's consider the case $s\in(1/4,1/2]$. Then $2s\in(1/2,1]$. In this case, the second condition in \eqref{potentials1}, that is $\displaystyle \|\partial_t^{1-\frac1{2s}}(\nu*P_s)\|_{*,p}\lesssim 1$, has $1-\frac1{2s}\in(-1,0)$ and can be rewritten as (via Fourier transform with respect to the $t$ variable)  $\displaystyle \|\nu*P_s*K\|_{*,p}\lesssim 1$  with $\displaystyle K(t)=\frac1{|t|^{2-\frac1{2s}}}$.
 \vv
 
Hence, the distributions $\nu$ admissible for the capacity $\gamma_\Theta^s(E)$, $0<s<1/2$, will be the ones supported on $E$ and satisfying
\begin{equation}
 \label{potentials2}
 \displaystyle \|R_j^s(\nu*P_s)\|_\infty\leq 1,\quad 1\leq j\leq N\quad\mbox{and}\quad\|\nu*P_s*_t\;K\|_{*,p}\leq 1.
 \end{equation}
 Here $R_j^s$ is the Calder\'on-Zygmund operator with kernel $\displaystyle \frac{x_j}{|x|^{N+2s}}$, the symbol $*_t$ denotes convolution with respect to the $t$ variable and $K$ is the kernel 
 \begin{equation}\label{kernelk}
 K=\frac1{|t|^{\frac1M(M+1-\frac1{2s})}}*\overset{M)}{\cdots}*\frac1{|t|^{\frac1M(M+1-\frac1{2s})}},
 \end{equation}
 where $M=\left[\frac1{2s}\right]$ (for a real number $\lambda,$ $[\lambda]$ stands for its integer part). 
 \vv

  Notice that for $s\in(0,1/2)$, \begin{equation}\label{rj}(-\Delta)^s\vphi=\sum_{j=1}^NR_j^s(\partial_j\vphi).\end{equation}
 
 We say that a $\mathcal{C}^{M+1}$ function $\vphi$ is admissible for an $s-$parabolic cube $Q\subset\R^{N+1}$, $0<s<1/2$, (recall that we write $Q=Q_1\times I_Q$ with $Q\subset\R^N$ and $I_Q\subset\R$) if 
 \begin{equation}\label{admissiblex}
 \|\nabla_x\vphi\|_\infty\leq\frac1{\ell(Q)}
 \end{equation}
 and for $0\leq k\leq M+1$
 \begin{equation}\label{admissiblet}
 \|\partial_t^k\vphi\|_\infty\leq\frac1{\ell(I_Q)^k}=\frac1{\ell(Q)^{2sk}}.
 \end{equation}
 Now the growth condition in this case reads as follows.
\begin{lemma}\label{case01/2}
Let $s\in(0,1/2)$, $M=\left[\frac1{2s}\right]$ and set $\displaystyle K=\frac1{|t|^{\frac1M(M+1-\frac1{2s})}}*\overset{M)}{\cdots}*\frac1{|t|^{\frac1M(M+1-\frac1{2s})}}$.
Suppose that $\nu$ is a distribution satisfying $$\displaystyle \|R_j^s(\nu*P_s)\|_\infty\leq 1,\quad 1\leq j\leq N\quad\mbox{and}\quad\|\nu*P_s*_t\;K\|_{*,p}\leq 1.$$

Then, if $\vphi$ is an admissible function for an $s-$parabolic cube $Q\subset\R^{N+1}$, we have $$|\langle\nu,\vphi\rangle|\lesssim\ell(Q)^{N+2s-1}.$$
\end{lemma}
\vv

Notice that if $s\in(1/4,1/2]$, for example, then $2s\in(1/2,1]$ and $\displaystyle M=1$. In this case $\displaystyle K=\frac1{|t|^{2-\frac1{2s}}}$ and the condition $\displaystyle \|\nu*P_s*K\|_{*,p}\leq 1$ can be rewritten (via Fourier transform with respect to the $t$ variable) as $\displaystyle \|\partial_t^{1-\frac1{2s}}(\nu*P_s)\|_{*,p}\lesssim 1$ (here $1-\frac1{2s}\in(-1,0)$), which is the same condition as in theorem \ref{case1/21}.
\begin{proof}
\begin{align*}
|\langle\nu,\vphi\rangle|&=|\langle\nu,\Theta^s\vphi*P_s\rangle|\leq|\langle\nu*P_s,(-\Delta)^s\vphi\rangle|+|\langle\nu*P_s,\partial_t\vphi\rangle|=I_1+I_2.
\end{align*}
Using \rf{rj} and the $s-$admissibility of $\vphi$, $$I_1\leq\sum_{j=1}^N\|R_j^s(\nu*P_s)\|_{\infty}\int_Q|\partial_j\vphi|\;dm\lesssim\ell(Q)^{N+2s-1}.$$
To estimate $I_2$, write $Q=Q_1\times I_Q$, with $Q_1\subset\R^N$ and $I_Q\subset\R$, and let $c_Q$ be the center of the interval $I_Q$. 
We claim that $\displaystyle\partial_t\vphi=g*_tK,$ with $\int g\;dm=0$,
\begin{equation}\label{inn}
|g(x,t)|\lesssim\ell(Q)^{-1},\;\;t\in 4I_Q.
\end{equation}
and
\begin{equation}\label{outt}
|g(x,t)|\lesssim\frac{\ell(Q)^{2s}}{|t-c_Q|^{1+\frac1{2s}}},\;\;t\in(2I_Q)^c.
\end{equation}
Once \eqref{inn} and \eqref{outt} are available, since $\displaystyle \|\nu*P_s*K\|_{*,p}\leq 1$, arguing as in the proof of theorem \ref{case1/21},  we get
$$I_2=|\langle\nu*P_s,g*K\rangle|=|\langle\nu*P_s*K,g\rangle|\lesssim \ell(Q)^{N+2s-1}.$$
Indeed, since $\int g\;dm=0$, writing $f=\nu*P_s*K$, we have
\begin{align*}
I_2&=|c||\langle f,g\rangle|=\left|c\int(f-m_Qf)\;g\;dm\right|\\&\lesssim\int_{2Q}|f-m_Qf|\;|g|\;dm+\int_{\R^{N+1}\setminus 2Q}(f-m_Qf)\;g\;dm\\&=I_{21}+I_{22}.
\end{align*}
Due to \eqref{inn}, $$I_{21}\lesssim\|f\|_{*,p}\,\ell(Q)^{N+2s}\ell(Q)^{-1}\leq\ell(Q)^{N+2s-1}.$$
For $I_{22}$, we split the domain of integration in annuli. Write $A_i= 2^{i} Q\setminus 2^{i-1} Q$ for $i\geq1$.
Remark that for an $s$-parabolic cube $Q = Q_1\times I_Q$, we denote
$$2^i Q = 2^i Q_1 \times 2^{2si} I_Q,$$
so that $2^iQ$ is a parabolic cube too (notice that if $Q$ is centered at the origin and we consider the parabolic dilation $\delta_\lambda(x,t) = (\lambda x,\lambda^{2s}t)$, $\lambda>0$, we have $2^iQ = \delta_{2^i}(Q)$).
Then, using the decay of $g$ given by \rf{outt}, we get
\begin{equation}\label{eq34342}
I_{22}\lesssim\sum_{i=1}^\infty \frac{\ell(Q)^{2s}}{\ell(2^iQ)^{1+2s}}\left(\int_{A_i\cap \supp g}|f-m_{2^{i}Q}f|\,dm+ \int_{A_i\cap \supp g}|m_{2^{i}Q}f-m_{Q}f|\,dm\right).
\end{equation}
To estimate the first integral on the right hand side, observe that by definition 
$\supp g\subset Q_1\times \R$. Using H\"older's inequality with some 
exponent $q\in(0,\infty)$ to be chosen in a moment and the fact that
$f\in BMO_p$ (together with John-Nirenberg), then we get:
\begin{align*}
\int_{A_i\cap \supp g}|f-m_{2^{i}Q}f|\,dm 
&\leq \left(\int_{2^iQ}|f-m_{2^{i}Q}f|^q\,dm\right)^{\frac1q} \,m(\supp g\cap 2^iQ)^{\frac1{q'}}\\
& \lesssim \ell(2^iQ)^{\frac{N+2s}q}\,(\ell(Q)^N\,\ell(2^iQ)^{2s})^{\frac 1{q'}} \\&= \ell(2^iQ)^{\frac Nq+ 2s}\,\ell(Q)^{\frac N{q'}}.
\end{align*}
For the last integral on the right hand side of \rf{eq34342}, we write
$$\int_{A_i\cap\supp g}|m_{2^{i}Q}f-m_{Q}f|\,dm\lesssim i\,m(2^iQ\cap \supp g) \leq i\,\ell(Q)^N\,\ell(2^iQ)^{2s}.$$
Therefore,
\begin{align*}
I_{22}&\lesssim\sum_{i=1}^\infty \frac{\ell(Q)^{2s}}{\ell(2^iQ)^{1+2s}}\,\Big(
\ell(2^iQ)^{\frac Nq+ 2s}\,\ell(Q)^{\frac N{q'}} + i\,\ell(Q)^N\,\ell(2^iQ)^{2s}\Big)\\&=\frac{\ell(Q)^{2s}\ell(Q)^{N+2s}}{\ell(Q)^{1+2s}}\sum_{i=1}^\infty \,
\left(2^{i(\frac Nq-1)}+\frac i{2^i}\right).
\end{align*}
Choosing $q>N$, we get
$$I_{22}\lesssim \ell(Q)^{N+2s-1}.$$
We still have to show claims \eqref{inn} and \eqref{outt}. Notice that (by taking Fourier transform with respect to $t$), $$\partial_t\vphi=(\partial_t^{1/2})^{2M}\partial_t\vphi*_t\frac1{|t|^{\frac1{2s}-M}}*_tK=g*_tK,$$This identity can be understood better if we distinguish two cases.
\begin{enumerate}
\item Case $M$ even. We can write $$\partial_t\vphi=\partial^{M+1}_t\vphi*_t\frac1{|t|^{\frac1{2s}-M}}*_tK.$$
Set $\displaystyle g=\partial^{M+1}_t\vphi*_t\frac1{|t|^{\frac1{2s}-M}}$. Then $\displaystyle \partial_t\vphi=g*K$. The zero mean of $\partial_t\vphi$ (integrating with respect to $t$) implies that $|g(x,t)|$ decays at most like $\displaystyle |t|^{-1-\frac1{2s}}$ at infinity. To see this notice that for $t\notin 2I_Q$ we have
\begin{align}\label{out2}
|g(x,t)|&=\left|\left(\partial_u^{M+1}\vphi*\frac1{|u|^{\frac1{2s}-M}}\right)(x,t)\right|=\left|\int_{I_Q}\frac{\partial_u\vphi(x,u)}{|t-u|^{\frac1{2s}}}du\right|\nonumber\\&=\left|\int_{I_Q}\partial_u\vphi(x,u)\left(\frac1{|t-u|^{\frac1{2s}}}-\frac1{|t-c_Q|^{\frac1{2s}}}\right)du\right|\\&\lesssim\frac{\ell(I_Q)}{|t-c_Q|^{1+\frac1{2s}}}\int_{I_Q}|\partial_u\vphi(x,u)|du\lesssim\frac{\ell(I_Q)}{|t-c_Q|^{1+\frac1{2s}}}=\frac{\ell(Q)^{2s}}{|t-c_Q|^{1+\frac1{2s}}}.\nonumber
\end{align}
And for $t\in4I_Q$, \begin{equation*}\label{in2}|g(x,t)|\lesssim\|\partial^{M+1}_u\vphi\|_\infty\ell(I_Q)^{1-(\frac1{2s}-M)}\lesssim\ell(I_Q)^{-\frac1{2s}}=\ell(Q)^{-1}.\end{equation*}

\item Case $M$ odd.  We can write $$\partial_t\vphi=\partial_t^{\frac1{2s}-M}\partial_t^M\vphi*K=g*K,$$ 
the last equality being a definition for $g$. Notice that for $t\notin 2I_Q$, \begin{align*}\label{out3}
|g(x,t)|&=\left|\int_{\R}\frac{\partial_u^M\varphi(x,u)-\partial_u^M\varphi(x,t)}{|u-t|^{\frac1{2s}-M+1}}du\right|=\left|\int_{I_Q}\frac{\partial_u^M\varphi(x,u)}{|u-t|^{\frac1{2s}-M+1}}du\right|\nonumber\approx\left|\int_{I_Q}\frac{\partial_u\varphi(x,u)}{|u-t|^{\frac1{2s}}}du\right|.\end{align*}
So arguing as in \eqref{out2}, 
$$|g(x,t)|\lesssim\frac{\ell(Q)^{2s}}{|t-c_Q|^{1+\frac1{2s}}}.$$

And for $t\in 4I_Q$,
\begin{align*}
|g(x,t)|&\leq\int_{5I_Q}\frac{|\partial_u^M\varphi(x,u)-\partial_u^M\varphi(x,t)|}{|u-t|^{\frac1{2s}-M+1}}du+\|\partial_t^M\varphi\|_\infty\int_{(5I_Q)^c}\frac{du}{|u-s|^{\frac1{2s}-M+1}}=T_1+T_2.
\end{align*}
To estimate $T_1$, apply the mean value theorem and \eqref{admissiblet} to get
$$T_1\leq\frac1{\ell(I_Q)^{M+1}}\int_{5I_Q}\frac{du}{|u-t|^{\frac1{2s}-M}}\lesssim\ell(I_Q)^{-\frac1{2s}}=\ell(Q)^{-1}.$$
Integrating and using \eqref{admissiblet} we get
$$T_2\leq\frac1{\ell(I_Q)^M}\int_{(5I_Q)^c}\frac{du}{|u-t|^{\frac1{2s}-M+1}}\lesssim\ell(Q)^{-\frac1{2s}}=\ell(Q)^{-1}.$$
Therefore for $t\in 4I_Q$,
\begin{equation*}\label{in3}
|g(x,t)|\lesssim\ell(Q)^{-1}.
\end{equation*}
\end{enumerate}
Therefore the claims \eqref{inn} and \eqref{outt} are proved and also the lemma.
\end{proof}

Now we can state a result analogous to the first statement of theorem \ref{hausdorffcontent} but for $s\in (0,1/2)$. 

\begin{theorem}\label{hausdorffcontent2}
 Let $s\in (0,1/2)$. For every Borel set $E\subset\R^{N+1}$, $$\gamma_{\Theta,+}^s(E)\leq\gamma^s_{\Theta}(E)\lesssim {\mathcal H}^{N+2s-1}_{\infty}(E).$$ 
 \end{theorem}

The proof is analogous to the one of theorem \ref{hausdorffcontent} but using lemma \ref{case01/2} instead of lemma \ref{case1/21}. We leave it for the reader. Although we think that the second statement of theorem \ref{hausdorffcontent} should hold for the case $0<s<1/2$, we are still having some technical problems to show it.  
\vvv

\end{document}